\documentclass[11pt]{article}

\usepackage{a4wide}
\usepackage{amsmath}
\usepackage{amsthm}
\usepackage{array}
\usepackage{amssymb}
\usepackage{amsfonts}
\usepackage[english]{babel}
\usepackage{epsf}
\usepackage{epsfig}
\usepackage{graphicx}

\newcommand{\R}{\mathbb{R}}
\newcommand{\N}{\mathbb{N}}

\newcommand{\E}{\mathbb{E}}
\newcommand{\Z}{\mathbb{Z}}

\newcommand{\Sd}{\textup{Sd}}

\newcommand{\CW}{\textup{CW}}

\newcommand{\id}{\textup{id}}

\newcommand{\tT}{\mathcal{T}}
\newcommand{\B}{\mathcal{B}}

\newcommand{\mmC}{\bold{C}}

\newtheorem{thm}{Theorem}[section]
\newtheorem{lem}[thm]{Lemma}
\newtheorem{prop}[thm]{Proposition}
\newtheorem{cor}[thm]{Corollary}
\newtheorem{rem}[thm]{Remark}
\newtheorem{ex}[thm]{Example}
\newtheorem{defn}[thm]{Definition}

\begin{document}

\title{Tilings, packings and expected Betti numbers in simplicial complexes}

\author{Nerm{\accent95\i}n Salepc{\accent95\i} and Jean-Yves Welschinger}

\maketitle

\begin{abstract}
Let $K$ be a finite simplicial complex. We prove that the normalized expected Betti numbers of a random subcomplex in its $d$-th barycentric subdivision $\Sd^d(K)$ converge to universal limits as $d$ grows to $+\infty$. In codimension one, we use canonical filtrations of $\Sd^d(K)$ to upper estimate these limits and get a monotony theorem which makes it possible to improve these estimates given any packing of disjoint simplices in $\Sd^d(K)$. We then introduce a notion of tiling of simplicial complexes having the property that skeletons and barycentric subdivisions of tileable simplicial complexes are tileable.   
This enables us to tackle the problem: How many disjoint simplices can be packed in $\Sd^d(K), d\gg0?$ 
 
\vspace{0.5cm}


\end{abstract}

\section{Introduction}
\subsection{Average Betti numbers}\label{Sect_Intro1}
Let $K$ be a finite $n$-dimensional simplicial  complex, $n>0$, and
$\Sd(K)$ be its first barycentric subdivision. In \cite{SW2} we introduced, for every $k\in \{1,\ldots, n\}$, a construction of codimension $k$ subcomplexes  $V_\epsilon$ of $\Sd(K)$ parametrized by  the $(k-1)$-dimensional simplicial cochains $\epsilon \in C^{k-1}(K;\Z/2\Z)$, see Definition~\ref{Defn_V}. The finite set $C^{k-1}(K;\Z/2\Z)$
being canonically a probability  space, or rather equipped with a family of probability measures $(\mu_{\nu})_{\nu\in[0,1]}$, where $\mu_\nu$ is the product measure  for which the probability that  the cochain $\epsilon$ takes the value 0 on a given $(k-1)$-simplex of $K$ is $\nu$ (see \cite{SW2}), this construction  produces a random   variable $\epsilon\in C^{k-1}(K;\Z/2\Z)\mapsto  V_\epsilon \subset \Sd(K)$ which raises the following questions of random topology:  What is the expected topology  of $V_\epsilon$, e.g. its expected Betti numbers? How does this expected topology behave under large iterated barycentric subdivisions? For every $p\in \{0,\ldots, n-k\}$, we denote the $p$-th Betti number of $V_\epsilon$ by $b_p(V_\epsilon)=\dim H_p(V_\epsilon; \Z/2\Z)$  and its mathematical expectation after $d$ barycentric subdivision by $\E_{\nu,d}(b_p)=\int_{C^{k-1}(\Sd^d(K);\Z/2\Z)}b_p(V_\epsilon)d\mu_{\nu}(\epsilon)$, $d\geq 0$. We proved in \cite{SW2} that there exist  universal constants  $c_p^{\pm}(n,k),$ such that

$$c_p^-(n,k)\leq \liminf_{d\to +\infty}\frac{\E_{\nu,d}(b_p)}{(n+1)!^df_n(K)} \leq \limsup_{d\to +\infty}\frac{\E_{\nu,d}(b_p)}{(n+1)!^df_n(K)}\leq c_p^+(n,k),$$  where $f_n(K)$ denotes the number of $n$-dimensional simplices of $K$. These results produced  counterparts  in this combinatorial framework to the ones obtained in \cite{GW15} and \cite{GW162, GW14} on the expected Betti numbers of random real algebraic submanifolds of  real projective manifolds or random nodal sets in smooth manifolds respectively (see also \cite{LL, NS1,NS2}). Our first result is the following:

\begin{thm}\label{Thm_Ed} 
Let $n>0$, $k\in\{1,\ldots,n\}$ and $p\in\{0,\ldots,n-k\}$. Then, for every finite $n$-dimensional simplicial complex $K$ and every $\nu\in[0,1]$, the sequence $\big(\frac{\E_{\nu,d}(b_p)}{f_n(K)(n+1)!^d}\big)_{d\geq 0}$ converges and its limit does not depend on $K$. 

\end{thm}

We denote by $e_{p,\nu}(n,k)$ the limit given by Theorem~\ref{Thm_Ed}. Recall that with the exception of the case $p=0$ (see \cite{NS2}), the counterpart of
Theorem~\ref{Thm_Ed}   in the theory of random polynomials (\cite{GW15}) or random nodal sets (\cite{NS1, NS2, GW14, GW162, LL, SaWi}) is open.
In the case of the standard $n$-simplex $\Delta_n$, we likewise set $\tilde{b}_p(V_\epsilon)=\sum b_p(\Sigma_\epsilon)$, where the sum is taken over all connected components $\Sigma_\epsilon$ of $V_\epsilon$ which do not intersect the boundary of $\Delta_n.$
We then set $\E_{\nu,d}(\tilde{b}_p)=\int_{C^{k-1}(\Sd^d(\Delta_n);\Z/2\Z)}\tilde{b}_p(V_\epsilon)d\mu_{\nu}(\epsilon)$ and get:

\begin{thm}\label{Thm_btilde}
Let $n>0$, $k\in\{1,\ldots, n\}$ and $p\in \{0, \ldots, n-k\}.$ Then, for every $\nu\in[0,1]$, the sequence $\big(\frac{\E_{\nu,d}(\tilde{b}_p)}{(n+1)!^d}\big)_{d\geq 0}$ is increasing and bounded from above.

\end{thm}

We denote by $\tilde{e}_{p,\nu}(n,k)$ the limit of the sequence given by Theorem~\ref{Thm_btilde}. We do not know whether $e_{p,\nu}(n,k)=\tilde{e}_{p,\nu}(n,k)$
 or not, except when $p=0$. It is related to a problem of percolation which we introduce  and discuss in \S~\ref{Sect_Per}, see Theorem~\ref{Thm_Per} (or \cite{BG} for another related problem of percolation). Nevertheless we get:  
 
 \begin{thm}\label{Thm_eetilde}
 Under the hypothesis of Theorem~\ref{Thm_btilde}, $e_{p,\nu}(n,k)\geq \tilde{e}_{p,\nu}(n,k)>0$. Moreover, $e_{0,\nu}(n,k)= \tilde{e}_{0,\nu}(n,k).$
 \end{thm}
 The codimension one case $k=1$ plays a special role. Namely, for every $\epsilon\in C^0(K;\Z/2\Z)$,  
we alternatively define $V'_\epsilon\subset |K|$ such that  for every  simplex $\sigma\in K,$
$V'_\epsilon\cap \sigma$ is the convex hull of the middle points of the edges of $\sigma$ where $\epsilon$ is not constant, that is where the exterior derivative $d\epsilon$  does not vanish.
We proved in \cite{SW2} that the pairs $(K, V_\epsilon)$ and $(K, V'_\epsilon)$ are always  homeomorphic, in fact isotopic, see Proposition~2.2  of \cite{SW2}. When $K$ is the moment polytope of some toric manifold equipped with a convex triangulation, the hypersurfaces $V'_\epsilon$ coincide with the patchwork hypersurfaces introduced by  O.~Viro (see \cite{V2, V1} and Remark~2.4 of \cite{SW2}). These hypersurfaces  $V'_\epsilon$ inherit the structure 
of a $\CW$ complex, having a $p$-cell for every $(p+1)$-simplex $\sigma$ of $K$ on which $\epsilon$ is not constant, $p\in \{0,\ldots, n-1\},$ see Corollary~2.5 of \cite{SW2}. We denote by $C_*(V'_\epsilon; \Z/2\Z)$ the cellular chain complex of $V'_\epsilon.$ Moreover, the $p$-cell $V'_\epsilon\cap \sigma$ is isomorphic to a product of two simplices:  the simplex spanned by vertices of $\sigma$ where $\epsilon=0$ and the simplex spanned by vertices of $\sigma$ where $\epsilon=1$.

We now observe that each $\epsilon\in C^0(K;\Z/2\Z)$  induces a filtration $\emptyset\subset K_0^\epsilon\subset K_1^\epsilon\subset \ldots \subset K_{[\frac{n+1}{2}]}^\epsilon=K$, where for every $i\in \{0,\ldots,[\frac{n+1}{2}]\},$ 
$$K_i^\epsilon=\{\sigma\in K\;|\; \#\epsilon_{|_\sigma}^{-1}(0)\leq i \mbox{ or }  \#\epsilon_{|_\sigma}^{-1}(1)\leq i \}.$$ This $i$-th subcomplex $K_i^\epsilon$ of $K$ is thus the union of all simplices that $V'_{\epsilon}$ meets along the empty set if $i=0$ or along a product of two simplices, one of which being of dimension $\leq i-1$ if $i>0.$ Then:

\begin{thm}\label{Thm_canisom}
For every finite $n$-dimensional simplicial complex $K$, $n>0$, and every $\epsilon \in C^0(K;\Z/2\Z),$ the relative simplicial chain complex $C_*(K,K_0^{\epsilon};\Z/2\Z)$ and the shifted cellular chain complex $C_*(V'_\epsilon;\Z/2\Z)^{[1]}$ are canonically isomorphic.  
\end{thm}  

We set, for every $p\in\{0,\ldots, n-1\}$ and $d\geq 0$, $\E_{\nu,d}(b_p(K_0^\epsilon))=\int_{C^0(\Sd^d(K);\Z/2\Z)}b_p(K_0^\epsilon)d\mu_\nu(\epsilon)$ and will write $\E_\nu$ instead of $\E_{\nu,0}$ for simplicity. We deduce:

\begin{cor}\label{Cor_canisom} For every finite $n$-dimensional simplicial complex $K$, $n>0$, and every $p\in\{0,\ldots, n-1\}$ and $\nu\in[0,1],$
$\E_{\nu} (b_p(K_0^\epsilon))-b_p(K)\leq \E_{\nu} (b_p(V_\epsilon))\leq \E_{\nu}(b_p(K_0^\epsilon))+b_{p+1}(K).$

In particular,  $$\lim_{d\to +\infty}\frac{\E_{\nu,d}(b_p(K_0^\epsilon))}{f_n(K)(n+1)!^d}=e_{p,\nu}(n,1).$$
\end{cor}

This result makes it possible to improve the upper estimates of $\E_{\nu,d}(b_p)$ or  of $e_{p,\nu}(n,k)$ given in \cite{SW2} in the case $k=1$ and to relate them with a packing problem in $K$ or $\Sd^d(K), d>0.$ To this end, for every finite simplicial complex $K$, every $p\in \N$ and $\nu\in[0,1]$, we set  $$M_{p,\nu}(K)=
{(\nu^{p+1}+(1-\nu)^{p+1})}{b_p(K)} +\nu(1-\nu)(\nu^{p}+(1-\nu)^{p})\sum_{i=p+1}^{\dim K}(-1)^{i+1-p}(f_i(K)-b_i(K)),$$
where $f_i(K)$ denotes the number of $i$-dimensional simplices of $K$ and the sum vanishes if $p\geq \dim K$. The positivity of $M_{p,\nu}(K)$ is closely related to the Morse inequalities associated to the simplicial chain complex  of $K$. Our key result is then the following monotony theorem, (see Theorem~\ref{Thm_Mp}).

\begin{thm}\label{Thm_MpL} Let $p\in\N$ and $\nu\in[0,1].$
 For every finite simplicial complex $K$ and every subcomplex $L$ of $K$, 
 $0\leq M_{p,\nu}(L)-\E_{\nu}(b_p(L_0^\epsilon))\leq M_{p,\nu}(K)-\E_{\nu}(b_p(K_0^\epsilon)).$ 
 
 \end{thm}
 
When $L=\emptyset$, Theorem~\ref{Thm_MpL}  combined with Corollary~\ref{Cor_canisom} already improves the upper estimates given in Corollary~4.2 of \cite{SW2} for $k=1$. But these get improved further whenever the left hand side in Theorem~\ref{Thm_MpL} is positive. This turns out to be the case when $L$ is a packing of disjoint simplices in $K$, or more generally of simplices which intersect along faces  of dimensions less than $p-1,$ since $\E_{\nu}(b_p(L_0^\epsilon))$ vanishes in this case as soon as $p>0$, see Proposition~\ref{Prop_L}. Theorem~\ref{Thm_MpL} combined with  Proposition~\ref{Prop_L} thus raises  the following  packing problem which we tackle in the second part of the paper, independent from the first one: How many disjoint simplices can be packed in the finite simplicial complex $K$? What about the asymptotic of such a maximal packing in $\Sd^d(K)$, $d\gg0$?
Indeed, let us denote by $\mathcal{L}_d^{n,p}$ the finite set of packings of simplices in $\Sd^d(\Delta_n)$ which intersect each other and the boundary of $\Sd^d(\Delta_n)$ along faces of dimensions less than $p-1,$ where $n,d>0$ and $p\in\{0,\ldots, n-1\}.$ We set $\lambda_{p,\nu}^d(n)=\frac{1}{(n+1)!^d}\max_{L\in \mathcal{L}_d^{n,p}}M_{p,\nu}(L).$
This sequence $(\lambda_{p,\nu}^d(n))_{d\geq 0}$ is increasing and bounded from above, see Proposition~\ref{Prop_lambda},
and we denote by $\lambda_{p,\nu}(n)$ its limit as $d$ grows to $+\infty$. We deduce the following asymptotic result.

\begin{thm}\label{Thm_limEd}
For every $n>0$, $\nu\in[0,1]$ and $p\in\{1,\ldots, n-1\},$

$$e_{p,\nu}(n,1)\leq \nu(1-\nu)(\nu^p+(1-\nu)^p)\sum_{l=p+1}^n(-1)^{l+1-p}q_{l,n}-\lambda_{p,\nu}(n).$$

\end{thm}

In Theorem~\ref{Thm_limEd}, $q_{l,n}$ denotes the asymptotic face number $\lim_{d\to +\infty}\frac{f_l(\Sd^d(\Delta_n))}{(n+1)!^d}$, see \cite{BW, DPS,SW1}. We do not know the actual value of $\lambda_{p,\nu}(n)$, but the results of the second part of this paper make it possible to estimate $\lambda_{p,\nu}(n)$ from below, see Theorem~\ref{Thm_mnp}.
Is it possible to likewise improve the upper estimates of \cite{GW161,GW162}?  What would then play the role of these packings?

\subsection{Tilings and Packings}
For every positive dimension $n$ and every $s\in\{0,\ldots,n+1\}$, we define the tile $T^n_s$ to be the complement of $s$ facets in the standard $n$-simplex $\Delta_n$. In particular, $T_0^n=\Delta_n$ and $T^n_{n+1}=\stackrel{\circ}{\Delta}_n$,  the interior of $\Delta_n$. An $n$-dimensional simplicial complex $K$ is called tileable  when $|K|$ can be  covered by disjoint $n$-dimensional tiles. For instance,  the boundary $\partial \Delta_{n+1}$ of the  standard  $(n+1)$-simplex is tileable and it has a  tiling which uses each tile $T_s^n$ exactly once, $s\in\{0,\ldots, n+1\},$ see Corollary~\ref{Cor_dDelta}. The $h$-vector $h(\tT)=(h_0(\tT),\ldots, h_{n+1}(\tT))$ of a finite tiling $\tT$ encodes the number of times $h_s(\tT)$ each tile $T_s^n$ is used in the tiling, $s\in\{0,1,\ldots, n+1\}.$
 We observe the following (see Theorem~\ref{Thm_fvsh}):
 
 \begin{thm} \label{Thm_hvTile}
 Let $K$ be a tileable  finite $n$-dimensional simplicial complex. Then, two tilings $\tT$ and $\tT'$ of $K$ have the same $h$-vector provided $h_0(\tT)=h_0(\tT').$ When $h_0(\tT)=1$, it coincides with the $h$-vector of $K$.
 \end{thm}

 Theorem~\ref{Thm_hvTile} thus provides a geometric interpretation of the $h$-vector for tileable finite $n$-dimensional  simplicial complexes which have a tiling $\tT$  such that $h_0(\tT)=1.$  When $K$ is connected, we call  such a tiling regular. Note that such an interpretation  was known for shellable simplicial complexes, a closely related notion, see \S~\ref{Sect_Tiling}.  Tileable simplicial    complexes  have the following key property (see Proposition~\ref{Prop_Sqtte} and Corollary~\ref{Cor_SdK}).
 
 \begin{thm}\label{Thm_SdSq} Let $K$ be a tileable  $n$-dimensional simplicial complex. Then, all its skeletons   and  barycentric subdivisions are tileable. Moreover, any tiling of $K$ induces a tiling on its skeletons  and barycentric subdivisions.
 \end{thm}

 We denote by $\Sd(\tT)$ the tiling of $\Sd(K)$ induced by a tiling $\tT$ of $K$.
 It was proved by A.~Bj\"orner \cite{B} that the barycentric subdivision of a shellable simplicial complex is shellable. Our proof seems more geometric. We actually prove that the first barycentric subdivision $\Sd(T_s^n)$ of each tile $T_s^n$ is tileable, $s\in\{0,\ldots, n+1\}$, see Theorem~\ref{Thm_SdPieces}. 
 We then study the matrix $H_n$ of size $(n+2)\times (n+2)$ whose  rows  are  the  $h$-vectors of the tilings of  $\Sd(T_s^n),$ $s\in\{0,\ldots, n+1\}.$ Let $\rho_n$ be the involution $(h_0,\ldots, h_{n+1})\in \R^{n+2}\mapsto (h_{n+1},\ldots, h_1, h_0)\in\R^{n+2}.$
 We prove the following:
 
 \begin{thm}
 For every $n>0$, $H_n$ is diagonalizable with eigenvalues $s!$, $s\in \{0,\ldots, n+1\}$. Moreover, it commutes with $\rho_n$ and the restriction of $\rho_n$ to the eigenspace of $s!$ is $(-1)^{n+1-s}\id.$  
 \end{thm} 
 
 We denote by  $h^n=(h_0^n,\ldots, h^n_{n+1})$ the eigenvector of the transposed matrix $H_n^t$ associated to the eigenvalue $(n+1)!$ and normalize it in such a way that $|h^n|=\sum_{s=0}^{n+1}h^n_s=1.$ We prove that $h_0^n=h^n_{n+1}=0$ and that $\rho_n(h^n)=h^n$, see Corollary~\ref{Cor_$h^n$}. Moreover:
 
 \begin{thm}
 Let $K$  be a finite $n$-dimensional simplicial complex equipped with a tiling $\tT$. Then, the sequence $\frac{1}{|h(\tT)|(n+1)!^d}h(\Sd^d(\tT))$ converges to $h^n$ as $d$ grows to $+\infty.$ Moreover, the matrix $\frac{1}{(n+1)!^d} H_n^d$ converges to  $(1,\ldots,1)(h^n)^t$ as $d$ grows to $+\infty.$
 \end{thm}

Hence, the asymptotic $h$-vector of a tiled finite $n$-dimensional simplicial complex $K$ does not depend on $K$ and equals $h^n$. This asymptotic result for $h$-vectors  has to be compared with the asymptotic of the $f$-vectors obtained in \cite{BW} (see also \cite{DPS} and \cite{SW1}).
This notion of tiling makes it possible to study the packing problem in tileable simplicial complexes or  rather in their  first barycentric subdivisions. In particular, we prove the following:

\begin{thm} \label{Thm_h0h1}Let $K$  be a finite $n$-dimensional simplicial complex equipped with a tiling $\tT$. Then, it is possible to pack $h_0(\tT)+h_1(\tT)$ disjoint $n$-simplices in $\Sd(K).$ Moreover, this packing can be completed by $h_{n+1-j}(\tT)+2^{n-1-j}\sum_{s=0}^{n-1-j}\frac{h_s(\tT)}{2^s}$ disjoint $j$-simplices for every $j\in\{0,\ldots, n-1\}.$  
\end{thm}

Since all barycentric subdivisions of the standard simplex are tileable, we are finally able to deduce in the limit the following lower estimates.

\begin{thm}\label{Thm_mnp} For every $n\geq 2$ and $p\in\{1,\ldots, n-1\}$, $$\lambda_{p,\nu}(n)\geq \frac{\nu(1-\nu)(\nu^p+(1-\nu)^p)}{(n+1)!}\Big[\Big(h_p^{n}+2^{p-1}\sum_{i=0}^{p-1}\frac{h^n_i}{2^i}\Big)\binom{n}{p+1}+\sum_{j=p+1}^{n+1}\Big(h_j^n+2^{j-2}\sum_{i=0}^{j-2}\frac{h^n_i}{2^i}\Big)\binom{n+p-j}{p+1}\Big].$$
\end{thm}

\vspace{0.5 cm}
\textbf{Acknowledgement:}
The second author is partially supported by the ANR project MICROLOCAL (ANR-15CE40-0007-01).

\tableofcontents

 \section{Asymptotic  behavior of the expected Betti numbers}
In the sequel, we denote by $K^{[i]}$ the set of $i$-dimensional simplices of a simplicial complex $K$ and by $K^{(i)}$ its $i$-skeleton. We first prove 
Theorems~\ref{Thm_btilde} and \ref{Thm_eetilde} and then Theorem~\ref{Thm_Ed}.
In \S~\ref{Sect_Per}, we study a problem of percolation related to these results.

\subsection{Proofs of Theorems~\ref{Thm_btilde} and \ref{Thm_eetilde}}

Given a simplicial  complex $K$, a $q$-simplex of $\Sd(K)$ is of the form $[\hat{\sigma}_0,\ldots, \hat{\sigma}_q]$, where $\hat{\sigma}_i$ denotes the barycenter of the simplex ${\sigma}_i$ of $K$, $ i\in\{0,\ldots, q\},$ and where ${\sigma}_i$ is a proper face of $\sigma_{i+1}, i\in\{0,\ldots, q-1\}.$  The block $D(\sigma)$ dual to a simplex $\sigma\in K$
is the union of all open simplices $[\hat{\sigma}_0,\ldots, \hat{\sigma}_q]$ of $\Sd(K)$ such that $\sigma_0=\sigma.$ The union of closed such simplices is denoted by $\overline{D}(\sigma)$, see \cite{M}. We recall the following definition from \cite{SW2}.

\begin{defn}\label{Defn_V}
Let $K$ be an $n$-dimensional simplicial complex, $n>0$, and let  $k\in\{1,\ldots, n\}$. For every $\epsilon\in C^{k-1}(K; \Z/2\Z)$, we denote by $V_\epsilon$ the subcomplex of $\Sd(K)$ dual to the cocycle $d\epsilon$, where $d:C^{k-1}(K; \Z/2\Z)\to C^{k}(K; \Z/2\Z)$  denotes the coboundary operator. Hence, $V_\epsilon$ is the union of the blocks $\overline{D}(\sigma)$ dual to the $k$-simplices $\sigma \in K$ such that $\langle d\epsilon, \sigma \rangle \neq 0$.
\end{defn}

\begin{proof}[Proof of Theorem~\ref{Thm_btilde}.]   For every $ m\in\{0,\ldots,d\},$  we observe that 
$$\begin{array}{rcl}
\E_{\nu,d}(\tilde{b}_p)&=&\int_{C^{k-1}(\Sd^d(\Delta_n);\Z/2\Z)}\tilde{b}_p(V_\epsilon)d\mu_{\nu}(\epsilon)\\
&\geq& \int_{C^{k-1}(\Sd^d(\Delta_n); \Z/2\Z)}\left(\sum_{\sigma\in \Sd^{d-m}(\Delta_n)^{[n]}} \tilde{b}_p(V_\epsilon\cap\sigma)\right)d\mu(\epsilon)\\
&=&\sum_{\sigma\in \Sd^{d-m}(\Delta_n)^{[n]}} \int_{C^{k-1}(\Sd^m(\sigma); \Z/2\Z)}\tilde{b}_p(V_\epsilon)d\mu(\epsilon), \, \mbox{ as $\mu_\nu$ is a product measure,}\\
&=& (n+1)!^{d-m}\E_{\nu,m}(\tilde{b}_p), \, \mbox{as $f_n(\Sd^{d-m}(\Delta_n))=(n+1)!^{d-m}.$}
\end{array}
$$

The sequence $\left(\frac{\E_{\nu,d}(\tilde{b}_p)}{(n+1)!^d}\right)_{d\in \N}$ is thus increasing. Moreover,  it is bounded from above by the converging sequence $\left(\frac{f_p(\Sd^{d+1}(\Delta_n))}{(n+1)!^d}\right)_{d\in\N}$, since $V_\epsilon$ is a subcomplex of $\Sd^{d+1}(\Delta_n),$ see \cite{BW, DPS} or \cite{SW1}. Hence the result.
\end{proof}

\begin{rem} \label{Rem_cp}For every $m\geq1$, let  $\mmC(m)$ be the finite set of homeomorphism classes of pairs $(\R^n,\Sigma),$
where $\Sigma$ is a closed manifold of dimension $n-k$ embedded in $\R^n$ by an embedding of complexity $m$, see Definition~{5.2}  and \S~5.3 of \cite{SW2}. With the notations 
of \S~5.3 of \cite{SW2}, we deduce the following lower estimate.

$$\begin{array}{rcl}
\tilde{e}_{p,\nu}(n,k)&\geq &\frac{\E_{\nu,d}(\tilde{b}_p)}{(n+1)!^d},\,\mbox{for every $d\geq 0$ by Theorem~\ref{Thm_btilde},}\\
&=&\sum_{m=1}^{d}\sum_{\Sigma\in \mmC(m)} b_p(\Sigma)\left( \frac{\E_{\nu,d}(N_\Sigma)}{(n+1)!^d}\right)\\
&\geq& \sum_{m=1}^{d}\sum_{\Sigma\in \mmC(m)} b_p(\Sigma)c_\Sigma,\,\mbox{by Theorem~5.11 of \cite{SW2}}.
\end{array}$$

By taking the  limit  as $d$ grows to $+\infty$, we deduce that $\tilde{e}_{p,\nu}(n,k)\geq c_p^-(n,k)$ by Definition~5.12 of \cite{SW2}.
\end{rem}

\begin{proof}[Proof of Theorem~\ref{Thm_eetilde}.]
By definition, for every $\epsilon\in C^{k-1}(\Sd^d(\Delta_n);\Z/2\Z)$ and $p\in\{0,\ldots,n-k\},$  $\tilde{b}_p(V_\epsilon)\leq {b}_p(V_\epsilon)$ so that after integration over  $C^{k-1}(\Sd^d(\Delta_n);\Z/2\Z)$, we get $\E_{\nu,d}(\tilde{b}_p)\leq \E_{\nu,d}({b}_p)$.
The inequality $\tilde{e}_{p,\nu}(n,k)\leq e_{p,\nu}(n,k)$ is deduced from Theorems~\ref{Thm_Ed} and~\ref{Thm_btilde} by passing to the limits as $d$ tends to $+\infty$, after dividing by $(n+1)!^d$.

When $p=0,$ we remark that by definition, for every $\epsilon \in C^{k-1}(\Sd^d(\Delta_n);\Z/2\Z),$ $b_0(V_\epsilon)\leq \tilde{b}_0(V_\epsilon)+b_0(V_\epsilon\cap \partial \Delta_n),$ as the number of connected components of $V_\epsilon$ which meet $\partial\Delta_n$ is bounded from above by $b_0(V_\epsilon\cap \partial \Delta_n).$ We deduce the inequality

\begin{align} \frac{\E_{\nu,d}(b_0)}{(n+1)!^d}\leq  \frac{\E_{\nu,d}(\tilde{b}_0)}{(n+1)!^d}+ \frac{\E_{\nu,d}(b_0(V_\epsilon\cap \partial \Delta_n))}{(n+1)!^d} \label{Eqn_Ed}.\end{align}

However, if $k=n,$ $\E_{\nu,d}(b_0(V_\epsilon\cap \partial \Delta_n))=0$ and so the result follows. If $k< n,$
$$\begin{array}{rcl}
\E_{\nu,d}(b_0(V_\epsilon\cap \partial \Delta_n))&=& \int_{C^{k-1}(\Sd^d(\Delta_n);\Z/2\Z)}b_0(V_\epsilon\cap \partial \Delta_n)d\mu_\nu(\epsilon)\\
&=&\int_{C^{k-1}(\Sd^d(\partial\Delta_n);\Z/2\Z)}b_0(V_\epsilon)d\mu_\nu(\epsilon), \,\mbox{as $\mu_\nu$ is a product measure},\\
&=&O(n!^d),\,\mbox{from Theorem~1.2 of \cite{SW2}}.
\end{array}$$

By letting $d$ tend to $+\infty$, we thus deduce from (\ref{Eqn_Ed})  that $e_{0,\nu}(n,k)\leq\tilde{e}_{0,\nu}(n,k)$.  Hence the result.
\end{proof}

\begin{rem}\label{Rem_eetilde}
The equality  $e_{0,\nu}(n,k)=\tilde{e}_{0,\nu}(n,k)$ is similar to \cite{NS2}. Moreover,  Theorem~\ref{Thm_eetilde} combined with Remark~\ref{Rem_cp} provides  the lower bound $e_{p,\nu}(n,k)\geq c_p^-(n,k)$ given by Corollary~5.13 of \cite{SW2}.
\end{rem}

Theorem~\ref{Thm_eetilde} raises the following question: Are the limits $e_{p,\nu}(n,k)$ and $\tilde{e}_{p,\nu}(n,k)$ equal in general or not? This question is related to a problem of percolation that we discuss in \S~\ref{Sect_Per}.

\subsection{Proof of Theorem~\ref{Thm_Ed}}

In order to prove Theorem~\ref{Thm_Ed}, we first need several preliminary results.
For every $m\in\{0,\ldots, d\}$, we consider $$A_{d,m}=\bigcup_{\sigma\in {\Sd^{d-m}(K)}^{[n]}}\stackrel{\circ}{\sigma}\mbox { and } B_{d,m}=\bigcup_{\tau\in \Sd^m(\Sd^{d-m}(K)^{(n-1)})}D(\tau),$$ so that $A_{d,m}, B_{d,m}$ are two open sets covering the underlying topological space $|K|$.  For every $q$-simplex $\theta=[\hat{\theta}_0,\ldots, \hat{\theta}_q]\in \Sd^{d+1}(K),$  let $l_m(\theta)\in\{-1,\ldots, q\}$ be the greatest integer $i$ such that $\theta_i\in \Sd^m(\Sd^{d-m}(K)^{(n-1)}),$ where $l_m(\theta)=-1$ if and only if there is no such  integer $i$.
When $l_m(\theta)\neq -1$, we moreover set 

$$\begin{array}{rlcl}
r_\theta:& \big(\theta\setminus [\hat{\theta}_{l_m(\theta)+1},\ldots, \hat{\theta}_q]\big)\times[0,1]&\to& \theta\setminus [\hat{\theta}_{l_m(\theta)+1},\ldots, \hat{\theta}_q] \\
&((1-\alpha)y+\alpha z,t)&\mapsto& (1-t\alpha)y+t\alpha z,
\end{array}$$  

where $\alpha \in[0,1[$,  $y\in [\hat{\theta}_{0},\ldots, \hat{\theta}_{l_m(\theta)}]$ and $z\in[\hat{\theta}_{l_m(\theta)+1},\ldots, \hat{\theta}_q]$.

\begin{prop}\label{Prop_Retract}
Let $K$ be a finite simplicial complex of positive dimension $n$ and let $\theta=[\hat{\theta}_0,\ldots, \hat{\theta}_q]$ be a $q$-simplex of $\Sd^{d+1}(K),$
$q\in \{0,\ldots,n\}$, $d\geq0$. Then for every $m\in\{0,\ldots,d\},$
\begin{enumerate}
\item $B_{d,m}\cap \theta=\theta\setminus  [\hat{\theta}_{l_m(\theta)+1},\ldots, \hat{\theta}_q]$ where  $[\hat{\theta}_{l_m(\theta)+1},\ldots, \hat{\theta}_q]=\emptyset$ if $l_m(\theta)=q.$
\item If $l_m(\theta)\neq -1,$ $r_\theta$  retracts $\theta\setminus [\hat{\theta}_{l_m(\theta)+1},\ldots, \hat{\theta}_q]$ by deformation onto $[\hat{\theta}_{0},\ldots, \hat{\theta}_{l_m(\theta)}].$
\item For every face $\sigma\in \theta$ such that  $l_m(\sigma)\neq -1$, the restriction of $r_\theta$ to $\sigma$ is $r_\sigma$.
\end{enumerate}
\end{prop}

\begin{proof}
By definition, $B_{d,m}$ is the union of the open simplices $\sigma$ of $\Sd^{d+1}(K)$ such that $l_m(\sigma)\neq -1.$ An open face  of $\theta$  is thus included in $B_{d,m}$
if and only if  it  contains a vertex $\hat{\theta}_i$ with $i\leq l_m(\theta).$  Hence the first part. The second and third parts follow from the definition of $r_\theta.$ 
\end{proof}

\begin{cor}\label{Cor_Retract}
Under the hypothesis of Proposition~\ref{Prop_Retract}, for every $m\in\{0,\ldots,d\},$
\begin{enumerate}
\item $B_{d,m}$ is an open subset of $|\Sd^{d+1}(K)|=|K|$ and   $\Sd^{d+1}(K)\setminus B_{d,m}$ is a subcomplex of $\Sd^{d+1}(K).$
\item For every subcomplex $L$ of $\Sd^{d+1}(K),$ the retractions $(r_\theta)_{\theta\in L}$  glue together to define $r_L : L\cap B_{d,m}\to L\cap B_{d,m}.$
Moreover,  $r_L$ retracts $L\cap B_{d,m}$ by deformation onto $L\cap\Sd^{m+1}(\Sd^{d-m}(K)^{(n-1)}).$
\item For every subcomplexes $M<L<\Sd^{d+1}(K),$ the restriction of $r_L$ to $M$ is $r_M$.
\end{enumerate}
\end{cor}

\begin{proof}
From the first part of Proposition~\ref{Prop_Retract}, the complement of $B_{d,m}$ in $\Sd^{d+1}(K)$ is the union of simplices $\theta$ such that $l_m(\theta)= -1.$    It is a subcomplex of $\Sd^{d+1}(K)$, which is closed. Hence the first part.
The third part of Proposition~\ref{Prop_Retract} guarantees that the retractions  $(r_\theta)_{\theta\in L}$ glue together to  define $r_L : L\cap B_{d,m}\to L\cap B_{d,m}$ and the second part of Proposition~\ref{Prop_Retract} guarantees that the latter is a retraction of $L\cap B_{d,m}$ to $L\cap\Sd^{m+1}(\Sd^{d-m}(K)^{(n-1)}).$ Finally, the last part of Corollary~\ref{Cor_Retract}  follows from the last part of Proposition~\ref{Prop_Retract}.
\end{proof}

\begin{prop}\label{Prop_Cn}
Under the hypothesis of Theorem~\ref{Thm_Ed}, there exists a universal constant $c(n)$ such that for every $m\in\{0,\ldots, d\},$
$$|\E_{\nu,d}(b_p)-\E_{\nu,d}(b_p(V_\epsilon\cap A_{d,m}))|\leq (n+1)!^df_n(K)\frac{c(n)}{(n+1)^m}+O(n!^d).$$
\end{prop}

\begin{proof}
It follows from Corollary~\ref{Cor_Retract} that for every $\epsilon\in C^{k-1}(\Sd^d(K); \Z/2\Z)$, $(A_{d,m}\cap V_\epsilon)\cup (B_{d,m}\cap V_\epsilon)$ is an open cover of $V_\epsilon$. The long exact sequence of Mayer-Vietoris associated to this open cover  (see \S~33 of \cite{M})  reads 
$$\ldots \to H_p(A_{d,m}\cap B_{d,m}\cap V_{\epsilon})\xrightarrow{i_p} H_p(A_{d,m}\cap V_{\epsilon})\oplus H_p(B_{d,m}\cap V_\epsilon)\xrightarrow{j_p}H_p(V_\epsilon)\xrightarrow{\partial_p} H_{p-1}(A_{d,m}\cap B_{d,m}\cap V_\epsilon)\to\ldots$$
 where the coefficients are in $\Z/2\Z$. We thus deduce that 
 
 $$\begin{array}{rcl}
b_p(V_\epsilon)&=&\dim \textup{Im}(\partial_p)+\dim  \textup{Im}(j_p)\\
&=&\dim  \textup{Im}(\partial_p)+b_p(A_{d,m}\cap V_\epsilon) +b_p(B_{d,m}\cap V_\epsilon) -\dim  \textup{Im}(i_p). 
\end{array}$$
After integration over $C^{k-1}(\Sd^d(K);\Z/2\Z)$, we get   

\begin{eqnarray}\label{Eqn_3term} |\E_{\nu,d}(b_p)-\E_{\nu,d}(A_{d,m}\cap V_\epsilon)| & \leq & \E_{\nu,d}(b_p(B_{d,m}\cap V_\epsilon)) +  \E_{\nu,d}(b_p(A_{d,m}\cap B_{d,m}\cap V_\epsilon)) \nonumber \\
&  & +  \; \E_{\nu,d}(b_{p-1}(A_{d,m}\cap B_{d,m}\cap V_\epsilon)).\end{eqnarray}

Let us now bound each term of the right hand side. We know
from Corollary~\ref{Cor_Retract} that for every  $\epsilon\in C^{k-1}(\Sd^d(K);\Z/2\Z),$ $r_{V_\epsilon}$ retracts $V_\epsilon \cap B_{d,m}$ by deformation onto $ V_\epsilon  \cap\Sd^{m+1}(\Sd^{d-m}(K)^{(n-1)}).$
Thus,

$$\begin{array}{rcl}
\E_{\nu,d}(b_p(V_\epsilon \cap B_{d,m}))&=&\E_{\nu,d}\big(b_p\big(V_\epsilon \cap \Sd^{m+1}(\Sd^{d-m}(K)^{(n-1)})\big)\big)\\
&\leq&f_p(\Sd^{m+1}(\Sd^{d-m}(K)^{(n-1)}))\\
&\leq& \sum_{i=p}^{n-1}\sum_{\sigma\in \Sd^{d-m}(K)^{[i]}}f_p(\Sd^{m+1}(\sigma))\\
&\leq& f_*(\Sd^{d-m}(K)) f_p(\Sd^{m+1}(\Delta_{n-1})),
\end{array}$$ 
where $f_*(\Sd^{d-m}(K))$ denotes the total face number of $\Sd^{d-m}(K).$

Let $\widetilde{K}$ be the subcomplex of $K$ made of simplices which are faces of $n$-simplices of $K$. Then, 

$$\begin{array}{rcl}
f_*(\Sd^{d-m}(K))&=&f_*(\Sd^{d-m}(\widetilde{K}))+f_*(\Sd^{d-m}(K\setminus \widetilde{K}))\\
&\leq&f_n(K) f_*(\Sd^{d-m}(\Delta_n))+f_*(\Sd^{d-m}(K\setminus \widetilde{K}))\\
&\leq& f_n(K) f_*(\Sd^{d-m}(\Delta_n))+O(n!^{d-m}),\\
\end{array}$$ 
from \cite{BW}, see also \cite{DPS, SW1}, since the simplices of $K\setminus \widetilde{K}$ are of dimension at most $(n-1).$ Likewise,  the sequences $\left(\frac{f_*(\Sd^{d-m}(\Delta_n))}{(n+1)!^{d-m}}\right)_{d\in\N}$ and $\left(\frac{f_p(\Sd^{m+1}(\Delta_{n-1}))}{n!^{m+1}}\right)_{m\in\N}$ are universal and convergent, see \cite{BW, DPS}. They are thus bounded by universal constants. We deduce the existence of a universal constant $c_1(n)$ such that 
$\E_{\nu,d}(b_p(B_{d,m}\cap V_\epsilon))\leq f_n(K)(n+1)!^d\frac{c_1(n)}{(n+1)^m}+O(n!^d).$

Similarly, from  Corollary~\ref{Cor_Retract} we know that for every $\epsilon\in C^{k-1}(\Sd^d(K);\Z/2\Z)$ and every $\sigma\in \Sd^{d-m}(K)^{[n]},$ $r_{V_\epsilon\cap\sigma}$ is a deformation retract of $B_{d,m}\cap\stackrel{\circ}{\sigma}\cap V_\epsilon$   onto $\partial \sigma\cap V_\epsilon.$
 So we have
 
 $$\begin{array}{rcl}
\E_{\nu,d}(b_p(A_{d,m}\cap B_{d,m} \cap V_\epsilon))&=&\sum_{\sigma\in \Sd^{d-m}(K)^{[n]}} \E_{\nu,d}(b_p(B_{d,m}\cap\stackrel{\circ}{\sigma}\cap V_\epsilon))\\
&=&\sum_{\sigma\in \Sd^{d-m}(K)^{[n]}} \E_{\nu,d}(b_p(V_\epsilon\cap\partial \sigma))\\
&\leq&\sum_{\sigma\in \Sd^{d-m}(K)^{[n]}} f_p(\Sd^{m+1}(\partial \sigma))\\
&\leq& f_n(K)\frac{(n+1)!^d}{(n+1)^m} \frac{f_p(\Sd^{m+1}(\partial \Delta_n)}{n!^m}.
\end{array}$$ 

Once again the sequence  $\left(\frac{f_p(\Sd^{m+1}(\partial\Delta_n))}{n!^{m}}\right)_{m\in \N}$ is universal and convergent and thus gets bounded by a universal constant.  The third term in (\ref{Eqn_3term}) is analogous to the second one, so it gets bounded by a universal constant as well. Hence the result.
\end{proof}

\begin{proof}[Proof of Theorem~\ref{Thm_Ed}.] 
We first prove the result for $K=\Delta_n$ and then deduce it for a general simplicial complex $K$. For every  $d\geq 0$, we set $$u_d=\frac{1}{(n+1)!^d}\int_{C^{k-1}(\Sd^d(\Delta_n);\Z/2\Z)} b_p(V_\epsilon) d\mu_\nu(\epsilon).$$
For every $\epsilon \in C^{k-1}(\Sd^d(\Delta_n);\Z/2\Z)$
and every $m\in\{0,\ldots, d\}$, $$b_p(V_\epsilon\cap A_{d,m})=\sum_{\sigma\in \Sd^{d-m}(\Delta_n)^{[n]}}b_p(V_\epsilon\cap \stackrel{\circ}{\sigma}).$$  After integration over $C^{k-1}(\Sd^d(\Delta_n);\Z/2\Z),$ we get $\E_{\nu,d}(b_p(V_\epsilon\cap A_{d,m}))=(n+1)!^{d-m}\E_{\nu,m}(b_p)$ since $\mu_\nu$ is a product measure, $\#\Sd^{d-m}(\Delta_n)^{[n]}=(n+1)!^{d-m}$ (see \cite{BW}) and  $b_p(V_\epsilon\cap \stackrel{\circ}{\sigma})=b_p(V_\epsilon\cap \sigma)$ by Corollary~\ref{Cor_Retract}.
Dividing by $(n+1)!^d$, we deduce from Proposition~\ref{Prop_Cn} the upper bound $$|u_d-u_m|\leq \frac{c(n)}{(n+1)^m}+O\Big(\frac{1}{(n+1)^d}\Big).$$
The sequence $(u_d)_{d\in \N}$ is thus a Cauchy sequence and so a converging sequence. The result  follows for $K=\Delta_n$.
Let us now suppose that $K$ is any finite simplicial complex of dimension $n$ and set, for every $d\geq 0$, $$v_d=\frac{1}{f_n(K)(n+1)!^d}\int_{C^{k-1}(\Sd^d(K);\Z/2\Z)} b_p(V_\epsilon) d\mu_\nu(\epsilon).$$ 
We deduce from Proposition~\ref{Prop_Cn}, after dividing by $f_n(K)(n+1)!^d$, that $|v_d-u_m|\leq \frac{c(n)}{(n+1)^m}+O\big(\frac{1}{(n+1)^d}\big)$. Since $(u_m)_{m\in \N}$ converges, this implies that $(v_d)_{d\in \N}$ converges as well and that $\lim_{d\to +\infty} v_d=\lim_{m\to +\infty} u_m.$ Hence the result.  
\end{proof}

\subsection{Percolation}\label{Sect_Per}
Let $n>0$ and $k\in\{1,\ldots, n\}$. For every $0\leq m\leq d$ and every $\sigma\in \Sd^{d-m}(\Delta_n)^{[n]}$, we denote by $P_{d,m}(\sigma)$ the set of $\epsilon\in C^{k-1}\big(\Sd^{m}(\Sd^{d-m}(\Delta_n)\setminus \stackrel{\circ}{\sigma})\big)$ for which there exists a path in $V_\epsilon$  connecting $\partial \Delta_n$ to $\partial \sigma$. In other words, $\epsilon \in P_{d,m}(\sigma)$ if and only if $V_\epsilon\setminus \stackrel{\circ}{\sigma}$ percolates between the boundaries of $\Delta_n$ and $\sigma.$

\begin{thm}\label{Thm_Per}
Let $\nu\in[0,1].$ If there exist $n>0$, $k\in\{1,\ldots,n\}$ and $p\in\{0,\ldots, n-k\}$ such that $e_{p,\nu}(n,k)>\tilde{e}_{p,\nu}(n,k),$ then, $n\geq 3,$ $p\geq 1$ and $\lim_{m\to +\infty}\mu_\nu(P_{m+m',m}(\sigma))=1$, for every $\sigma\in\Sd^{m'}(\Delta_n)^{[n]}, m'\in \N$.
\end{thm}

\begin{proof}
If $p=0$,  $e_{p,\nu}(n,k)=\tilde{e}_{p,\nu}(n,k)$ for every $k\in\{1,\ldots, n\}$ and  $n>0$, by Theorem~\ref{Thm_eetilde}. If $n=2$, we need to prove the equality for $k=p=1.$ 
In this case, for every $\epsilon\in C^0(\Sd^d(\Delta_2);\Z/2\Z),$ $V_\epsilon$ is a homological manifold
of dimension one by Theorem~1.1 of \cite{SW2}. The connected components of $V_\epsilon$  which meet $\partial \Delta_2$ are thus homeomorphic to $\Delta_1$ so that $b_1(V_\epsilon)=\tilde{b}_1(V_\epsilon)$. The equality $e_1(2,1)=\tilde{e}_1(2,1)$ is  then obtained by integrating over $C^0(\Sd^d(\Delta_2);\Z/2\Z).$

Let us suppose now that there exists $n>0$, $k\in \{1,\ldots, n\}$ and $p\in\{0,\ldots, n-k\}$ such that $e_{p,\nu}(n,k)>\tilde{e}_{p,\nu}(n,k)$. Let $m'\in \N$ and $\sigma\in \Sd^{m'}(\Delta_n)^{[n]}.$ For every $m\geq1$, we denote by $\widetilde{P}_{m'+m}(\sigma)$ the set of $\tilde{\epsilon}\in C^{k-1}(\Sd^{m'+m}(\Delta_n)\setminus \Sd^m(\sigma))$ for which there exists $\epsilon \in P_{m'+m}(\sigma)$ whose restriction to the complement of $\sigma$ is $\tilde{\epsilon}$ ($\Sd^{m+m'}(\Delta_n)\setminus \Sd^m(\sigma)$ is not a simplicial complex. It is enough to show that $\lim_{m\to +\infty}\mu_\nu(\widetilde{P}_{m+m',m}(\sigma))=1.$
Indeed, let $\tilde{\sigma}\in \Sd^{2}(\sigma)$ which  does not meet the boundary of $\sigma$.
Then, for every $m\geq1$ and $\tilde{\epsilon}\in \widetilde{P}_{m+m'+2,m}(\tilde{\sigma})$, the restriction of $\tilde{\epsilon}$ to $\Sd^{m+2}(\Sd^{m'}(\Delta_n)\setminus \stackrel{\circ}{\sigma})$ lies in $P_{m+m'+2,m+2}(\sigma),$ since a path in $V_\epsilon$ connecting $\partial \Delta_n$ to $\partial \tilde{\sigma}$ has to cross $\partial \sigma$. Let $\widehat{P}_{m+m'+2,m+2}(\tilde{\sigma})$ be the image of $\widetilde{P}_{m+m'+2,m+2}(\tilde{\sigma})$ in $P_{m+m'+2,m+2}(\sigma)$ by this restriction map, so that $\mu_\nu(\widehat{P}_{m+m'+2,m+2}(\tilde{\sigma}))\leq \mu_\nu({P}_{m+m'+2,m+2}(\sigma))$. Since $\mu_\nu$ is a product measure,  $\mu_\nu(\widehat{P}_{m+m'+2,m+2}(\tilde{\sigma}))\geq \mu_\nu(\widetilde{P}_{m+m'+2,m+2}(\tilde{\sigma}))$ and we deduce  the inequality $\mu_\nu(\widetilde{P}_{m+m'+2,m+2}(\tilde{\sigma}))\leq \mu_\nu({P}_{m+m'+2,m+2}({\sigma}))$. Thus, the fact that $\lim_{m\to +\infty}\mu_\nu(\widetilde{P}_{m+m'+2,m+2}(\tilde{\sigma}))=1$ implies that $\lim_{m\to +\infty}\mu_\nu({P}_{m+m'+2,m+2}({\sigma}))=1$.

Let us now prove that $\lim_{m\to +\infty}\mu_\nu(\widetilde{P}_{m+m',m}({\sigma}))=1$, a result which does not depend on the choices of $m'$ and $\sigma$. For every $\epsilon \in C^{k-1}(\Sd^{m'+m}(\Delta_n)),$ let $\Lambda_\epsilon$
be the union of the connected components of $V_\epsilon$ which meet $\partial \Delta_n$, so that $b_p(V_\epsilon)=\tilde{b}_p(V_\epsilon)+b_p(\Lambda_\epsilon)$. Replacing  $V_\epsilon$ by $\Lambda_\epsilon$ in Proposition~\ref{Prop_Cn}, we deduce that

$$\lim_{m\to +\infty}\frac{1}{(n+1)!^{m'+m}}|\E_{m'+m}(b_p(\Lambda_\epsilon))-\E_{m'+m}(b_p(\Lambda_\epsilon\cap A_{m+m',m}))|=0.$$

Under our hypothesis we then deduce that $$\lim_{m\to +\infty}\frac{1}{(n+1)!^{m'+m}}\E_{m'+m}(b_p(\Lambda_\epsilon\cap A_{m+m',m}))= e_{p,\nu}(n,k)-\tilde{e}_{p,\nu}(n,k)>0.$$
However, by definition,

 $$\begin{array}{rcl}
 \E_{m+m'}(b_p(\Lambda_\epsilon\cap A_{m+m',m}))&=&\sum_{\sigma\in \Sd^{m'}(\Delta_n)^{[n]}}\int_{C^{k-1}(\Sd^{m+m'}(\Delta_n)\setminus \Sd^m(\sigma))}\int_{C^{k-1}(\Sd^m(\sigma))} b_p(\Lambda_\epsilon\cap \stackrel{\circ}{\sigma})d\mu_\nu(\epsilon)\\
 &=&\sum_{\sigma\in \Sd^{m'}(\Delta_n)^{[n]}}\int_{\widetilde{P}_{m+m',m}(\sigma)}\int_{C^{k-1}(\Sd^m(\sigma))} b_p(\Lambda_\epsilon\cap \stackrel{\circ}{\sigma})d\mu_\nu(\epsilon)\\
&\leq & \sum_{\sigma\in \Sd^{m'}(\Delta_n)^{[n]}}\mu_\nu({\widetilde{P}_{m+m',m}(\sigma)})\E_{\nu,m}(b_p-\tilde{b}_p),
\end{array}$$
so that 

$$\frac{1}{(n+1)!^{m'}} \sum_{\sigma\in \Sd^{m'}(\Delta_n)^{[n]}}\mu_\nu({\widetilde{P}_{m+m',m}(\sigma)})\geq \frac{\E_{\nu,m'+m}(b_p(\Lambda_\epsilon\cap A_{m+m',m}))}{(n+1)!^{m'}\E_{\nu, m}(b_p-\tilde{b}_p)}.$$

From what precedes, the right hand side converges to 1 as $m$ grows to $+\infty.$ We thus deduce that  for every $m'\in \N$ and every $\sigma\in \Sd^{m'}(\Delta_n)^{[n]},$ $\lim_{m\to+\infty}\mu_\nu(\widetilde{P}_{m+m',m}(\sigma))=1.$ Hence the result.
\end{proof}

\begin{rem}
\begin{enumerate}
\item It would be interesting to prove that for every $m'\in\N$, the sequence  $\frac{1}{(n+1)!^{m'}}\sum_{\sigma\in \Sd^{m'}(\Delta_n)^{[n]}}\mu_\nu(\widetilde{P}_{m+m',m}(\sigma))$
which appeared at the end of the proof of Theorem~\ref{Thm_Per} converges to a limit  $u_{m'}$ as $m$ grows to $+\infty$,  without assuming that $e_{p,\nu}(n,k)>\tilde{e}_{p,\nu}(n,k).$  Then, the sequence $(u_{m'})_{m'\in \N}$ would be submultiplicative. 

\item It would also be interesting to know whether the sequence $\big(\mu_\nu(P_{m+m',m}(\sigma))\big)_{m\in \N}$ converges or not for every $m'$ and every  $\sigma\in \Sd^{m'}(\Delta_n)^{[n]},$ and whether or not the limit is 1 without assuming that $e_{p,\nu}(n,k)>\tilde{e}_{p,\nu}(n,k)$.  Nevertheless, if this condition is necessary by Theorem~\ref{Thm_Per} for $e_{p,\nu}(n,k)$ to be greater than $\tilde{e}_{p,\nu}(n,k)$, it is not sufficient, a priori.
This question of percolation seems already of interest in dimension two.

\end{enumerate}
\end{rem}

\section{Refined upper bounds}
\subsection{Induced filtrations}
Let $K$ be a finite $n$-dimensional simplicial complex. For every $\epsilon\in C^0(K;\Z/2\Z)$, $K$ inherits the filtration $$\emptyset \subset K^\epsilon_0\subset K^\epsilon_1\subset\ldots \subset K^\epsilon_{[\frac{n+1}{2}]}=K,$$ where $K^\epsilon_i=\{\sigma \in K| \# (\epsilon |_\sigma)^{-1}(0)\leq i \mbox{ or } \# (\epsilon_{|_\sigma})^{-1}(1)\leq i \}$,  $i\in\{0,1,\ldots,[\frac{n+1}{2}] \}$.

We denote by $C_*(K,K^\epsilon_0;\Z/2\Z)$ the associated relative simplicial chain  complex and by ${C}_*(V'_\epsilon;\Z/2\Z)$
the cellular chain complex of the CW-complex $V'_\epsilon$, see Corollary~{2.5} of \cite{SW2}.
We now prove Theorem~\ref{Thm_canisom} and Corollary~\ref{Cor_canisom}.

\begin{proof}[Proof of Theorem~\ref{Thm_canisom}.]
For every $p\in\{0,\ldots, n-k\}$, the relative chain complex $C_p(K,K_0^{\epsilon};\Z/2\Z)$  contains a canonical basis given by simplices on which $\epsilon$ is not constant.
Similarly, a basis of   ${C}_p(V'_\epsilon;\Z/2\Z)$ is given by a $p$-cell for every $(p+1)$-simplex of $K$ on which $\epsilon$ is not constant.
Hence the canonical isomorphism between $C_*(K,K_0^{\epsilon};\Z/2\Z)$ and ${C}_*(V'_\epsilon;\Z/2\Z)^{[1]}$.  
Moreover, the boundary operator  of ${C}_*(V'_\epsilon;\Z/2\Z)^{[1]}$ gets identified to the composition $pr_\epsilon\circ \partial$, where $\partial$ is the boundary operator of $C_*(K;\Z/2\Z)$ and $pr_\epsilon$ is the projection on $C_*(K,K_0^\epsilon;\Z/2\Z)$.
\end{proof}

Theorem~\ref{Thm_canisom} gives new insights on the following corollary which has been established in \cite{SW2}, see  also \cite{SW1}. Recall that for every finite simplicial complex of dimension $n$, $q_K(T)$ denotes the face polynomial $\sum_{p=0}^{n}f_p(K)T^p.$

\begin{cor}[Theorem~1.5 of \cite{SW2}]\label{Cor_q1/2}
For every  finite $n$-dimensional simplicial complex $K$ and every $\nu\in[0,1],$
$\chi(K)+\E_\nu(K)=\nu q_K(-\nu)+(1-\nu)q_K(\nu-1)$.
\end{cor}

\begin{proof}
For every $\epsilon \in C^0(K;\Z/2\Z)$, the short exact sequence $$0\to C(K_0^\epsilon)\xrightarrow{i}C(K)\xrightarrow{\pi}C(K,K_0^\epsilon)\to 0.$$
combined with Theorem~\ref{Thm_canisom} implies $\chi(K^\epsilon_0)-\chi(V_\epsilon)=\chi(K).$ 
After integration over all $\epsilon$, we get
$$\begin{array}{rcl}\chi(K)+\E(\chi)&=&\E(\chi(K_0^\epsilon))=\sum_{\sigma\in K}(-1)^{\dim \sigma}\mu\{ \epsilon\in C^0(K;\Z/2\Z)| \epsilon_{|_\sigma} =\mbox{constant}\}\\
&=&\sum_{p=0}^n(-1)(\nu^{p+1}+(1-\nu)^{p+1})f_p(K).\end{array}$$
\end{proof}

\begin{proof}[Proof of Corollary~\ref{Cor_canisom}.]
The short exact sequence $0\to C(K_0^\epsilon)\xrightarrow{i}C(K)\xrightarrow{\pi}C(K,K_0^\epsilon)\to 0$
induces  the long exact sequence in homology
$$\ldots \to H_{p+1}(K)\xrightarrow{{(\pi_{p+1}})_*}H_{p+1}(K, K_0^\epsilon)\xrightarrow{{(\partial_{p+1}})_*}H_p(K_0^\epsilon)\xrightarrow{({i_{p}})_*}H_p(K)\to \ldots.$$
We deduce,
$$\begin{array}{rcl}
b_{p+1}(K,K_0^\epsilon)&=&\dim \ker {(\partial_{p+1}})_*+\dim \textup{Im} {(\partial_{p+1}})_* \\
&=&\dim \textup{Im}{(\pi_{p+1}})_*+\dim (\ker {i_p})_*\\
&=&b_{p+1}(K)+b_{p}(K_0^\epsilon)-\dim \textup{Im} ({i_{p+1}})_*-\dim \textup{Im} ({i_{p}})_*\\
&\leq& b_{p+1}(K)+b_{p}(K_0^\epsilon).
\end{array}$$

The upper estimates follow using the isomorphism  $H_{p+1}(K,K_0^\epsilon)\to \tilde{H}_p(V_\epsilon)$ given by Theorem~\ref{Thm_canisom}. We furthermore deduce

$$\begin{array}{rcl}
b_{p+1}(K,K_0^\epsilon)&=&\dim \ker {(\partial_{p+1}})_*+ b_p(K_0^\epsilon)-\dim \textup{Im} ({i_{p}})_* \\
&=&\dim \ker ({\partial_{p+1}})_*+b_{p}(K_0^\epsilon)-b_p({K})+ \dim \textup{Im} ({\pi_{p}})_*\\
&\geq& b_{p}(K_0^\epsilon)-b_p(K).
\end{array}$$

The isomorphism  given by Theorem~\ref{Thm_canisom} now implies the lower estimate after integration over $C^0(K;\Z/2\Z).$ Finally, the asymptotic result follows from the definition of $e_{p,\nu}(n,1)$ and the invariance of the Betti numbers of $K$ under barycentric subdivisions. 
\end{proof}

Thanks to Corollary~\ref{Cor_canisom},  in order to estimate $\E_\nu(b_p)$ from above,  it suffices to upper estimate $\E_\nu(b_p(K_0^\epsilon)).$ A rough upper estimate is given by Proposition~\ref{Prop_bp<fp}, but it is going to be improved by Theorem~\ref{Thm_UpB}
and improved further by Corollary~\ref{Cor_Mp}.

\begin{prop}\label{Prop_bp<fp}
Let $K$ be a finite $n$-dimensional simplicial complex, $n>0$. For every $\nu\in[0,1]$ and every $p\in \{0,\ldots,n-1\}$,  $\E_\nu(b_p(K_0^\epsilon))\leq {f_p(K)}(\nu^{p+1}+(1-\nu)^{p+1})=\E_\nu(f_p(K_0^\epsilon)).$
\end{prop}

\begin{proof}
The result follows from the fact that  $b_p(K_0^\epsilon)\leq f_p(K_0^\epsilon)$ for every $\epsilon\in C^0(K;\Z/2\Z)$, combined with
$$\begin{array}{rcl}
\E(f_p(K_0^\epsilon))&=&\int_{C^0(K;\Z/2\Z)} \#\{\sigma\in K^{[p]}|\, \epsilon_{|_\sigma}  \textup{is constant}\}d\mu(\epsilon)\\
&=&\sum_{\sigma K^{[p]}}\mu\{\epsilon \in C^0(K;\Z/2\Z)| \epsilon_{|_\sigma}  \textup{is constant} \}\\
&=& {f_p(K)}(\nu^{p+1}+(1-\nu)^{p+1}).
\end{array}$$
\end{proof}

\begin{lem}\label{Lem_Sp}
Let $K$ be a finite $n$-dimensional simplicial complex, $n\geq0$.
For every $p\in \{0,\ldots, n\}$, the space of $p$-cycles $Z_p(K) \subset C_p(K)$ has a complement spanned by simplices of $K.$
\end{lem}

\begin{proof} We proceed by induction on the codimension of $Z_p(K)$. If $Z_p(K)$ contains all the simplices, there is nothing to prove.  Otherwise, we choose a simplex $\sigma\notin Z_p(K)$, so that $\langle\sigma\rangle\cap Z_p(K)=\{0\}.$ Then we consider the quotient $C_p(K)/\langle\sigma\rangle$ and the image of $Z_p(K)$ in this quotient. Its codimension decreases by one. The result follows by induction.
\end{proof}

The upper estimate given in Propositon~\ref{Prop_bp<fp} can now be improved, see \S~\ref{Sect_Intro1} for the definition of  $M_{p,\nu}(K)$. 
\begin{thm}\label{Thm_UpB}
Let $\nu\in[0,1]$ and $p\in\N$. Then, for every finite simplicial complex $K$,  $\E_\nu(b_p(K_0^\epsilon))\leq M_{p,\nu}(K).$
\end{thm}

\begin{proof}
By definition,  for every $\epsilon\in C^0(K;\Z/2\Z),$
$$ \begin{array}{rcl}
b_p(K_0^\epsilon)&=&\dim Z_p(K_0^\epsilon)- \dim B_p(K_0^\epsilon)\\
&=& f_p(K_0^\epsilon)-\dim B_{p-1}(K_0^\epsilon)-\dim B_p(K_0^\epsilon).
\end{array}
$$ 
Let $S_p(K)$ be a subspace complementary to $Z_p(K)$ in $C_p(K)$. Its dimension equals
$$\begin{array}{rcl}
\dim S_p(K)&=&f_p(K)-\dim Z_p(K)\\
&=&(f_p(K)-b_p(K))-\dim B_p(K)\\ 
&=&\sum_{i=p}^n(-1)^{k-p}(f_i(K)-b_i(K)), \mbox{ by induction.}
\end{array}$$

From Lemma~\ref{Lem_Sp} we can choose $S_p(K)$ to be spanned by simplices.
A linear combination of such simplices  belongs to $C_p(K_0^\epsilon)$ if and only if each simplex belongs to $C_p(K_0^\epsilon)=\ker pr_\epsilon$. Therefore,
$\dim (S_p(K)\cap C _p(K_0^\epsilon))=\#\{\sigma \in \B_p| \sigma\in K_0^\epsilon\},$ where $\B_p$ denotes our basis of $S_p(K)$. 
However, $S_p(K)\cap C_p(K^\epsilon_0)$ is transverse to $\ker \partial_p^{K_0^\epsilon}$ in $C_p(K_0^\epsilon),$ so that 
$\dim (S_p(K)\cap C_p(K_0^\epsilon))\leq \dim B_{p-1}({K_0^\epsilon})$. 
We deduce 

$$\begin{array}{rcl}
\E_\nu(b_p(K_0^\epsilon))&\leq&\E_\nu(f_p(K_0^\epsilon))-\E_\nu(\#\{\sigma\in \B_p| \sigma\in K_0^\epsilon\})-\E_\nu(\#\{\sigma\in \B_{p+1}| \sigma\in K_0^\epsilon\})\\
&\leq&(\nu^{p+1}+(1-\nu)^{p+1})(f_p(K) -\#\B_p)-(\nu^{p+2}+(1-\nu)^{p+2})\#\B_{p+1}\\
&\leq&(\nu^{p+1}+(1-\nu)^{p+1})\big(f_p(K) -\sum_{i=p}^{\dim K}(-1)^{i-p}(f_i(K)-b_i(K))\big)-\\
&&(\nu^{p+2}+(1-\nu)^{p+2})\sum_{i=p+1}^{\dim K}(-1)^{i+1-p}(f_i(K)-b_i(K))\\
&\leq& (\nu^{p+1}+(1-\nu)^{p+1})b_p(K)+\nu(1-\nu)(\nu^{p}+(1-\nu)^{p})\sum_{i=p+1}^{\dim K}(-1)^{i+1-p}f_i(K)-b_i(K).
\end{array}$$  
The  third inequality follows from the fact that $\#\B_p=\dim S_p(K)$.
\end{proof}

\subsection{Monotony theorem}\label{Sect_Refined}

\begin{lem} \label{Lem_Mpdelta}
For every $n>0$ and $p\in\{1,\ldots, n-1\}$, $M_{p,\nu}(\Delta_n)=\nu(1-\nu)(\nu^{p}+(1-\nu)^{p})\binom{n}{p+1}$, while $M_{0,\nu}(\Delta_n)=1 +\nu(1-\nu)2n.$ \end{lem}

\begin{proof}
The simplex $\Delta_n$ is contractible,  so that $b_0(\Delta_n)=1$ and $b_i(\Delta_n)=0$ for every $i>0.$
Moreover, for every $i\in\{0,\ldots, n\}$, $f_i(\Delta_n)=\binom{n+1}{i+1}$ as the $i$-simplices of $\Delta_n$ are in one-to-one correspondence with sets of $i+1$ vertices of $\Delta_n$.

Thus, for $p\in\{1,\ldots,n-1\},$

$$\begin{array}{rcl}\sum_{i=p+1}^n (-1)^{i+1-p}\binom{n+1}{i+1}&=&\sum_{i=p+1}^{n-1}(-1)^{i+1-p}(\binom{n}{i}+\binom{n}{i+1}) +(-1)^{n+1-p}\\
&=&(-1)^{n-p}+\binom{n}{p+1}+(-1)^{n+1-p}\\
&=&\binom{n}{p+1},
\end{array}$$
and the result follows from the definition of $M_{p,\nu}(\Delta_n)$, see \S~\ref{Sect_Intro1}.
\end{proof}

Let us remark that the left hand side of Theorem~\ref{Thm_UpB} vanishes for a large family of complexes, for instance those given by the following proposition.

\begin{prop}\label{Prop_L} Let $p>0$ and $L$ be a finite simplicial complex which contains a family of simplices $\{\sigma_i\in L, i\in I\}$ such that
\begin{enumerate}
\item $\forall \tau \in L, \exists i\in I \mbox{ such that }  \tau<\sigma_i$,
 \item  $\forall i\neq j \in I, \dim(\sigma_i\cap \sigma_j)< p-1$. 
  \end{enumerate}
  Then, for every $ l\geq p$, $b_p(L)=0=\E_\nu(b_p(L_0^\epsilon)).$
\end{prop}

\begin{proof} Let $D=\bigsqcup_{i\in I}\sigma_i$ be the disjoint union of the $\sigma_i$, $i\in I$, and $h:D\to L$  be the associated canonical simplicial  map.  From the hypothesis, for every $ l\geq p-1, \,$ $h:D\to L$ provides a bijection between the $l$-simplices of $D$  and $L$. Thus, $h_\#: C_l(D)\to C_l(L)$ is an isomorphism of vector spaces and a chain map which induces an isomorphism $h_*: H_l(D)\to H_l(L)$. As $H_l(D)=0$, we deduce  that $b_l(L)$ vanishes. 
Now, for every $\epsilon \in C^0(L,\Z/2\Z)$, we set $\tilde{\epsilon}=\epsilon\circ  h.$ Then, $h: D_0^{\tilde{\epsilon}}\to L_0^\epsilon$  induces another isomorphism  $h_*: H_l(D_0^{\tilde{\epsilon}})\to H_l(L_0^\epsilon)$ for $l\geq p$. However, by definition, $D_0^{\tilde{\epsilon}}$ is again a disjoint union of simplices, so that  $H_l(D_0^{\tilde{\epsilon}})=0$. Hence the result.
\end{proof}
The following monotony theorem completes Theorem~\ref{Thm_MpL}.

\begin{thm}\label{Thm_Mp} Let $\nu\in[0,1]$ and $p\in \N$. For every finite simplicial complex $K$ and every subcomplex $L$ of $K$,
$M_{p,\nu}(L)-\E_\nu(b_p(L_0^\epsilon))\leq M_{p,\nu}(K)-\E_\nu(b_p(K_0^\epsilon)).$ 

Similarly, $\E_\nu(f_p(L))-\E_\nu(b_p(L_0^\epsilon))\leq\E_\nu(f_p(K))-\E_\nu(b_p(K_0^\epsilon)).$
 \end{thm}
 
 \begin{proof}
 Let $S_p(L)$ (resp. $S_{p+1}(L)$) be  complementary to $Z_p(L)$ (resp. $Z_{p+1}(L)$) in $C_p(L)$ (resp. $C_{p+1}(L)$) and spanned by $p$-simplices, see Lemma~\ref{Lem_Sp}. The intersection of $S_p(L)$ with $Z_p(K)$ is thus $\{0\}$ and $Z_p(K)\oplus S_p(L)=Z_p(K) +C_p(L). $ Let $S_p(K,L)$  (resp. $S_{p+1}(K,L)$) be the complement
of this space in $C_p(K)$ (resp. in $C_{p+1}(K))$ spanned by $p$-simplices of $K$  in such a way that $S_p(K,L)$ completes $S_p(L)$ to a complement $S_p(K)=S_p(L) \oplus S_p(K,L)$  of $Z_p(K)$ in $C_p(K)$ (resp. $S_{p+1}(K,L)$ completes $S_{p+1}(L)$ to a complement $S_{p+1}(K)=S_{p+1}(L) \oplus S_{p+1}(K,L)$  of $Z_{p+1}(K)$ in $C_{p+1}(K)$). 

Let $B_p(K)\subset C_p(K)$ be the image of  $\partial_{p+1}: C_{p+1}(K)\to C_p(K)$. For every $\epsilon\in C^0(K;\Z/2\Z),$  $L_0^\epsilon=K_0^\epsilon\cap L,$ so that $B_{p-1}(L_0^\epsilon)\subset B_{p-1}(K_0^\epsilon)$  and $B_p(L_0^\epsilon)\subset B_p(K_0^\epsilon)$. 
 Moreover, $S_p(K,L)\cap C_p(K_0^\epsilon)$ is complement to $Z_p(K_0^\epsilon)+C_p(L_0^\epsilon)$ in $C_p(K_0^\epsilon),$  so that 
 $\dim (S_p(K,L)\cap C_p(K_0^\epsilon))+\dim B_{p-1}(L_0^\epsilon) \leq \dim B_{p-1}(K_0^\epsilon)$ and similarly
 $ \dim (S_{p+1}(K,L)\cap C_{p+1}(K_0^\epsilon))+\dim B_p(L_0^\epsilon) \leq \dim B_p(K_0^\epsilon).$
As before we deduce 
 
  $\begin{array}{rcl}
  b_p(K_0^\epsilon)&=&f_p(K_0^\epsilon)-\dim B_{p-1}(K_0^\epsilon)-\dim B_p(K_0^\epsilon)\\
  &\leq& f_p(K_0^\epsilon)-f_p(L_0^\epsilon)+\big(f_p(L_0^\epsilon)-\dim B_{p-1}(L_0^\epsilon)-\dim B_p(L_0^\epsilon)\big)-\\
  &&\dim (S_p(K,L)\cap C_p(K_0^\epsilon)) -\dim (S_{p+1}(K,L) \cap C_{p+1}(K_0^\epsilon))\\
  &\leq& f_p(K_0^\epsilon)-\dim (S_p(K)\cap C_p(K_0^\epsilon)) -\dim (S_{p+1}(K) \cap C_{p+1}(K_0^\epsilon))+b_p(L_0^\epsilon)-\\
  &&\big(f_p(L_0^\epsilon)-\dim (S_p(L)\cap C_p(L_0^\epsilon)) -\dim (S_{p+1}(L) \cap C_{p+1}(L_0^\epsilon))\big).
  \end{array}$
  
The first part of the result is thus obtained as in Theorem~\ref{Thm_UpB} by integrating over $C^0(K;\Z/2\Z)$ and the second part follows from the inequality $\dim S_p(K)\cap C_p(K_0^\epsilon)\geq \dim S_p(L)\cap C_{p}(L_0^\epsilon)$ (resp. $\dim S_{p+1}(K)\cap C_{p+1}(K_0^\epsilon)\geq \dim S_{p+1}(L)\cap C_{p+1}(L_0^\epsilon)$).
 \end{proof}
 
 Thanks to Proposition~\ref{Prop_L}, Theorem~\ref{Thm_Mp} makes it possible to improve  the upper bounds given by Theorem~\ref{Thm_UpB}.

 \begin{cor}\label{Cor_Mp}
 Under the hypothesis of Theorem~\ref{Thm_Mp}, let $L$
 be subcomplex of $K$ satisfying the hypothesis of Proposition~\ref{Prop_L}. Then, $$\E_\nu(b_p(K_0^\epsilon))\leq (\nu^{p+1}+(1-\nu)^{p+1})b_p(K)+\nu(1-\nu)(\nu^p+(1-\nu)^p)\sum_{i=p+1}^{\dim K}(-1)^{i+1-p}\big(f_i(K)-f_i(L)-b_i(K)\big).$$
  \end{cor}

\begin{proof}
From Theorem~\ref{Thm_Mp}, $\E_\nu(b_p(K_0^\epsilon))\leq M_{p,\nu}(K)-M_{p,\nu}(L)+\E_\nu(b_p(L_0^\epsilon))$ and from Proposition~\ref{Prop_L}, $\E_\nu(b_p(L_0^\epsilon))=0$ and $M_{p,\nu}(L)=\nu(1-\nu)(\nu^p+(1-\nu)^p)\sum_{i=p+1}^n(-1)^{i+1-p}f_i(L).$ Hence the result.
\end{proof}

For every $d>0$, let $\mathcal{L}_d^{n,p}$ be the finite set of simplicial subcomplexes $L$ of $\Sd^d(\Delta_n)$
containing a family of simplices $\{\sigma_i\in L, i\in I\}$ such that 

\begin{enumerate}
\item $\forall \tau\in L$, $\exists i\in I$ such that $\tau<\sigma_i,$
\item $\forall i\neq j\in I, \dim(\sigma_i\cap\sigma_j)<p-1,$
\item $\forall i\in I, \dim(\sigma_i\cap  \Sd^d(\partial\Delta_n))<p-1$.
\end{enumerate}
We set $\lambda_{p,\nu}^d(n)=\frac{1}{(n+1)!^d}\textup{max}_{L\in \mathcal{L}_d^{n,p}}M_{p,\nu}(L)$.
 
 \begin{prop}\label{Prop_lambda}
 The sequence $(\lambda_{p,\nu}^d(n))_{d\geq 0}$ is increasing and bounded.
 \end{prop}
 
 Let $\lambda_{p,\nu}(n)$ be the limit of this sequence $(\lambda_{p,\nu}^d(n))_{d\geq 0}$.

 \begin{proof}
Let $d>0$. The set $ \mathcal{L}_d^{n,p}$ being finite, there exists a subcomplex $L_d$ in $\mathcal{L}_d^{n,p}$ which maximize $M_{p,\nu}$ over $\mathcal{L}_d^{n,p}$. Let $m>0$. For every $n$-simplex $\sigma$ of $\Sd^m(\Delta_n)$, we choose a simplicial  isomorphism $\Delta_n\xrightarrow{f_\sigma}\sigma$. Let $L_{d+m}=\bigcup_{\sigma\in {\Sd^m(\Delta_n)}^{[n]}}(f_\sigma)_*L_d$. It is a subcomplex of $\Sd^{d+m}(\Delta_n).$ Moreover,
 $L_{d+m}$ belongs to $\mathcal{L}_{d+m}^{n,p},$ so that $\frac{1}{(n+1)!^{d+m}}M_{p,\nu}(L_{d+m})\leq \lambda^{d+m}_{p, \nu}(n).$
 From  Proposition~\ref{Prop_L},
$M_{p,\nu}(L_{d+m})=\nu(1-\nu)(\nu^p+(1-\nu)^p)\sum_{i=p+1}^n(-1)^{i+1-p}f_i(L_{d+m}).$
However, by construction,  $L_{d+m}\cap \Sd^d(\Sd^m(\Delta_n)^{(n-1)})$  is of dimension $<p-1$ 
so that it does not contribute to the computation of $M_{p,\nu}(L_{d+m}).$
As the cardinality of $\Sd^m(\Delta_n)^{[n]}$ is $(n+1)!^m$, we deduce that  $M_{p,\nu}(L_{d+m})=(n+1)!^mM_{p,\nu}(L_d).$ Thus $\lambda_{p,\nu}^d(n)\leq \lambda^{d+m}_{p,\nu}(n)$. 

Now, for every $L\in \mathcal{L}_d^{n,p},$ $$M_{p,\nu}(L)\leq M_{p,\nu}(\Sd^d(\Delta_n))=\nu(1-\nu)(\nu^p+(1-\nu)^p)\sum_{i=p+1}^{n}(-1)^{i+1-p}f_i(\Sd^d(\Delta_n)).$$
The sequence $\frac{M_{p,\nu}(\Sd^d(\Delta_n))}{(n+1)!^d}$ is convergent, see \cite{BW, DPS}, so that the sequence $ \lambda_{p,\nu}^d(n)$ is bounded. Hence the result. 
 \end{proof}
 
 We can now deduce Theorem~\ref{Thm_limEd}.
Recall that by definition, for every $i\in\{0,\ldots, n\},$  $q_{i,n}=\lim_{d\to +\infty}\frac{f_i(\Sd^d(K))}{f_n(K)(n+1)!^d}.$

\begin{proof}[Proof of Theorem~\ref{Thm_limEd}.]
We proceed as in the proof of Proposition~\ref{Prop_lambda}.
For every $n$-simplex $\sigma\in K$, we choose a simplicial isomorphism $f_\sigma:\Delta_n\to \sigma$. For every $d>0,$ there exists a subcomplex $L_d\in  \mathcal{L}_d^{n,p}$ which maximize the function $M_{p,\nu}$ over $\mathcal{L}_d^{n,p}$. We set $K_d=\bigcup_{\sigma\in K^{[n]}}(f_\sigma)_*L_d.$ We deduce as in the proof of Proposition~\ref{Prop_lambda} that $M_{p,\nu}(K_d)=f_n(K)M_{p,\nu}(L_d).$ 
Since $b_{p+1}(\Sd(K))=b_{p+1}(K)$, we deduce 
  
  $\begin{array}{rcl}
  \E_{\nu,d}(b_p)&\leq& \E_\nu(b_p({\Sd^d(K)}^\epsilon_0))+b_{p+1}(K),\,\mbox{from Corollary~\ref{Cor_canisom},}\\
  &\leq& M_{p,\nu}(\Sd^d(K))- M_{p,\nu}(K_d)+b_{p+1}(K), \,\mbox{from  Proposition~\ref{Prop_L} and Theorem~\ref{Thm_Mp},}\\
   &\leq& M_{p,\nu}(\Sd^d(K))-f_n(K)M_{p,\nu}(L_d)+b_{p+1}(K).
   \end{array}$
  
  However, $$M_{p,\nu}(\Sd^d(K))={b_p(K)}{(\nu^{p+1}+(1-\nu)^{p+1})} +\nu(1-\nu)(\nu^{p}+(1-\nu)^{p})\sum_{i=p+1}^n(-1)^{i+1-p}(f_i(\Sd^d(K))-b_i(K)),$$ so that $\lim\limits_{d\to +\infty}\frac{M_{p,\nu}(\Sd^d(K))}{f_n(K)(n+1)!^d}=\nu(1-\nu)(\nu^{p}+(1-\nu)^{p})\sum_{i=p+1}^n(-1)^{i+1-p}q_{i,n}$, see \cite{BW, DPS} or  \cite{SW2}.
  By definition, $\lim\limits_{d\to +\infty}\frac{M_{p,\nu}(L_d)}{(n+1)!^d}=\lambda_{p,\nu}(n)$ and thus  $$\lim\limits_{d\to +\infty}\frac{E_d(b_p)}{f_n(K)(n+1)!^d}\leq \nu(1-\nu)(\nu^{p}+(1-\nu)^{p})\sum_{i=p+1}^n(-1)^{i+1-p}q_{i,n}-\lambda_{p,\nu}(n).$$
\end{proof}

In the light of Theorem~\ref{Thm_limEd}, it would be great to be able to compute  $\lambda_{p,\nu}(n)$. Theorem~\ref{Thm_mnp}  provides an estimate of this number from below.
\section{Tilings}\label{Sec_Tiling}

\subsection{Tiles}
For every $n>0$ and every $s\in \{0,1,\ldots,n+1\}$, we set $T_s^n=\Delta_n\setminus (\sigma_1\cup\ldots\cup\sigma_s),$ where $\sigma_i$ denotes a facet of $\Delta_n,$ $ i\in \{1,\ldots,s\}.$  In particular, the tile $T_{n+1}^n$ is the open $n$-simplex $\stackrel{\circ}{\Delta}_n$ and $T_0^n$ is the closed one $\Delta_n,$ see Figure~\ref{Fig_Tiles}.

\begin{figure}[h]
   \begin{center}
    \includegraphics[scale=0.4]{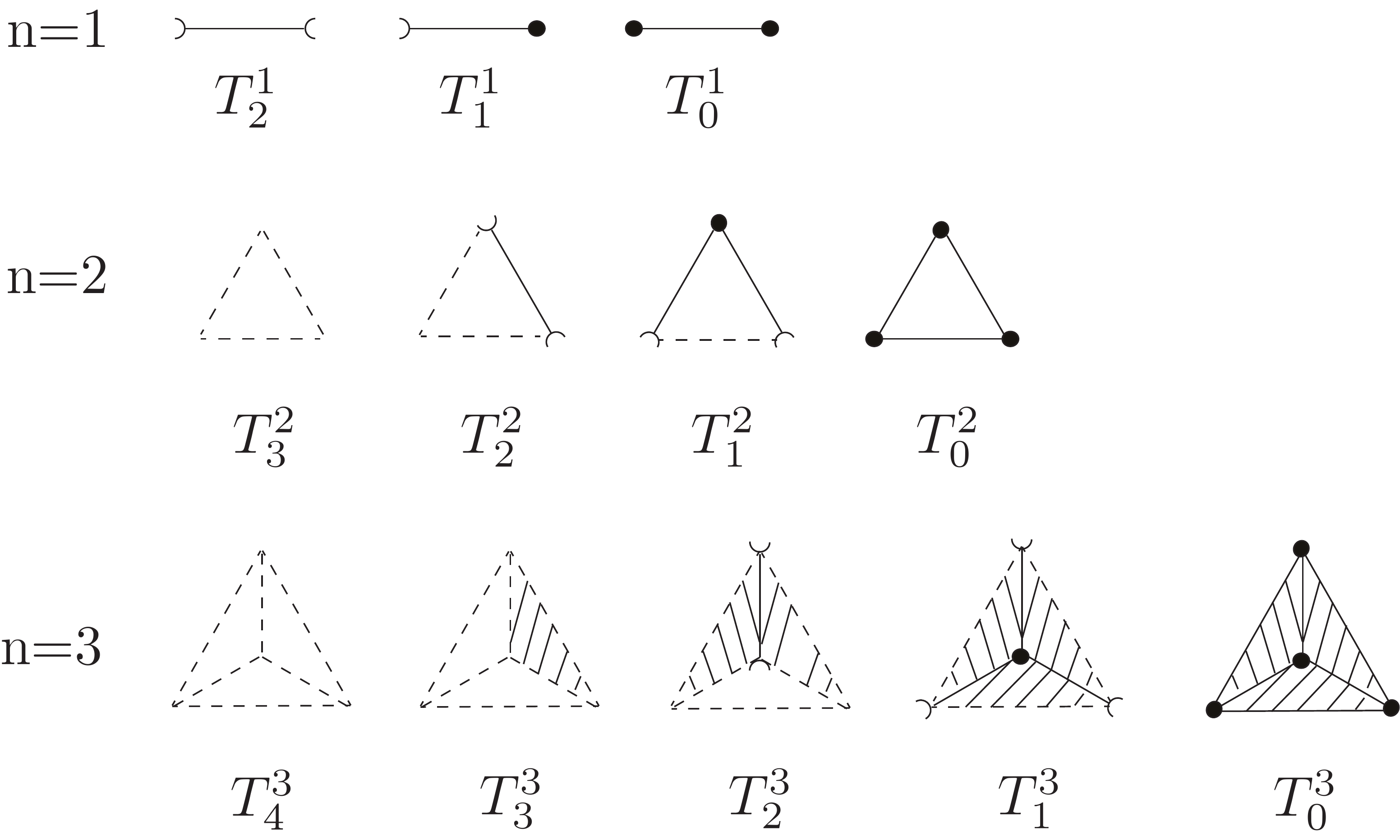}
    \caption{The tiles in dimensions less than 4.}
    \label{Fig_Tiles}
      \end{center}
 \end{figure}

\begin{prop}\label{Prop_Pieces}
For every $n>0$ and every $s\in \{0,\ldots, n+1\},$ $T^{n+1}_s$ is a cone over $T_s^n$, deprived of  its center if $s\neq 0$. Moreover, 
$T^{n+1}_s$ is a disjoint union  $T_{n+2}^{n+1}\sqcup T_s^n\sqcup T_{s+1}^n\sqcup\ldots \sqcup T^n_{n+1}.$
In particular,  the cone $T_s^{n+1}$ deprived of its base $T^n_s$ is $T_{s+1}^{n+1}$.
\end{prop}

\begin{proof}
If $s=0,$ $T_{0}^{n+1}=\Delta_{n+1}=c\ast T_{0}^n$, where $c$ denotes a vertex of $\Delta_{n+1}$. If $s>0$, $T_s^n=\Delta_n\setminus (\sigma_1\cup\ldots\cup\sigma_{s})$ by definition, where $\sigma_i$ is a facet for every $i\in\{1,\ldots, s\},$ and so  $(c\ast T^n_s)\setminus \{c\}=(c\ast \Delta_n)\setminus ((c\ast \sigma_1)\cup\ldots\cup (c\ast \sigma_{s})).$ However, $c\ast \Delta_n=\Delta_{n+1}$ and $\theta_i=c\ast \sigma_i$ is an $n$-simplex,  $i\in \{1,\ldots, s\}$. It follows from the definition that $T_{s}^{n+1}=(c\ast T_s^n)\setminus \{c\}.$ It is the cone over $T_s^n$ deprived of its center $c$. The base $T_s^n$  of this cone is the intersection of $T_{s+1}^{n+1}$ with the base $\theta=\Delta_n$ of the cone $\Delta_{n+1}=c\ast \Delta_n$. Thus, $T_{s}^{n+1}\setminus T_{s}^{n}=(c\ast \Delta_n)\setminus (\theta_1\cup\ldots\cup \theta_{s}\cup \theta)=T^{n+1}_{s+1}$.  The result holds true for $s=0$ as well, since by definition $T_{1}^{n+1}=\Delta_{n+1}\setminus \Delta_n=T_{0}^{n+1}\setminus T^n_{0}.$
By induction, we deduce that for every $s\in \{0,\ldots,n+1\},$ $T_{s}^{n+1}\cap \partial \Delta_{n+1}$ is the disjoint union $T_s^n\sqcup\ldots\sqcup T^n_{n+1}.$ 
\end{proof} 

\begin{cor}\label{Cor_dDelta}
For every $n>0$,  $\partial \Delta_{n+1}=\bigsqcup_{s=0}^{n+1}T_s^n.$
\end{cor}

\begin{proof} By definition $T_{0}^{n+1}=\Delta_{n+1}=\stackrel{\circ}{\Delta}_{n+1}\sqcup \partial \Delta_{n+1}=T_{n+2}^{n+1}\sqcup \partial \Delta_{n+1}.$ It  follows from Proposition~\ref{Prop_Pieces} that  $T_{0}^{n+1}=T_{n+2}^{n+1}\sqcup(\sqcup_{s=0}^{n+1}T_s^n).$ Hence the result.
\end{proof}

\begin{prop}\label{Prop_fi}
For every $n>0$ and every  $s\in \{0,\ldots, n+1\}$, 
$$f_j(T_s^n)=\begin{cases} 
0 & \mbox{if } 0\leq j<s-1,\\
\binom{n+1-s}{n-j}& \mbox{if } s-1\leq j\leq n.
\end{cases}$$ 
\end{prop}

By face number of a tile, we mean its number of open simplices of the corresponding dimension.

\begin{proof}
We proceed by induction on the dimension $n$.
If $n=1$, one checks the result. Now let us suppose that the result holds true for  $n\geq 1$. 
By definition, $T_{0}^{n+1}=\Delta_{n+1}$, so that for every $j\in \{0,\ldots, n+1\},$ $f_j(T_{0}^{n+1})=\binom{n+2}{j+1}=\binom{n+2}{n+1-j}$. Hence the result for $s=0.$ Similarly, $T_{n+2}^{n+1}=\stackrel{\circ}{\Delta}_{n+1}$, so that $f_j(T_{n+2}^{n+1})=\delta_{n+1 j}$. Now, if $1\leq s\leq n+1$, we know from Proposition~\ref{Prop_Pieces} that $T_{s}^{n+1}=(c\ast T_s^n)\setminus \{c\}$. The open faces of $T_{s}^{n+1}$ are thus either the faces of the basis $T_s^n$ of the cone, or the cones over the faces of $T_s^n$. We deduce from Pascal's formula that for every $j\in\{0,\ldots,n+1\},$ $f_j(T_{s}^{n+1})=f_j(T^n_s)+f_{j-1}(T_s^n)=\binom{n+1-s}{n-j}+\binom{n+1-s}{n+1-j}=\binom{n+2-s}{n+1-j}.$ 
\end{proof}

\subsection{Tilings}\label{Sect_Tiling}

For every $n>0$, let $\tT(n)$ be the set of  $n$-dimensional simplicial complexes that can be tiled by 
$T_0^n,\ldots,T_{n+1}^n.$ 

\begin{prop}\label{Prop_Sqtte}
\begin{enumerate}
\item Every pure  finite  simplicial complex of dimension one is tileable.
\item For every $n>0$, if $K\in \tT(n)$, then   for every $i\in \{1,\ldots, n\},$ the $i$-skeleton  $K^{(i)}$ belongs to $ \tT(i).$ Moreover, any tiling of $K$
 induces a tiling on $K^{(i)}.$ 
 \end{enumerate}
\end{prop}

Recall that an $n$-dimensional simplicial complex is called pure if each of its simplex is a face of an $n$-simplex. The second part of  Proposition~\ref{Prop_Sqtte} is the first part of Theorem~\ref{Thm_SdSq}.
 
\begin{proof}

For the first part we proceed by induction on the number of edges. If  such a complex $K$ contains only one edge,  it is tiled by $T_0^1=\Delta_1$, since it is pure. Now let us suppose that  the result holds true for every pure complex containing $N$ edges. Let $K$ be a pure simplicial complex of dimension one with $N+1$ edges and let $e$ be an edge of $K$. Thus, $K=K'\cup e$ where $K'$ is a pure simplicial complex covered by $N$ edges and by the hypothesis it can be tiled. We choose a tiling of $K'$. Then, $e$ has 0,1 or  2 common vertices with $K'$.  We then extend the tiling of $K'$ to a tiling of $K$ by adding  $T_0^1, T_1^1$ or $T_2^1,$ respectively.

And for the second part we proceed by induction on the dimension $n$. If $n=1$ the result follows from the first part. Let us suppose that the result holds true for the  dimension $n$. Let $K$  be a  tiled  finite simplicial complex of dimension $n+1$. The $n$-skeleton of $K$ can be obtained by removing all open $(n+1)$-simplices. 
These are exactly the interiors of the tiles of $K$. However, from Proposition~\ref{Prop_Pieces}, for every $s\in \{0,\ldots, n+2\},$ $T_s^{n+1}\setminus T_{n+2}^{n+1}$ is tiled by the tiles of dimension $n$.  Thus, the $n$-skeleton of $K$ gets an induced tiling of dimension $n$. The result follows from the fact that $K^{(i)}=(K^{(n)})^{(i)}.$
\end{proof}

Proposition~\ref{Prop_Sqtte} raises the following question. Let $X$ be  a triangulated manifold with boundary. If $X$  belongs to $\tT(n)$, does  $\partial X$ belong to $\tT(n-1)$?

\begin{ex} 
\begin{enumerate}

\item  Figure~\ref{Fig_extile} shows some examples of one-dimensional tiled simplicial complexes.
\begin{figure}[h]
   \begin{center}
    \includegraphics[scale=0.4]{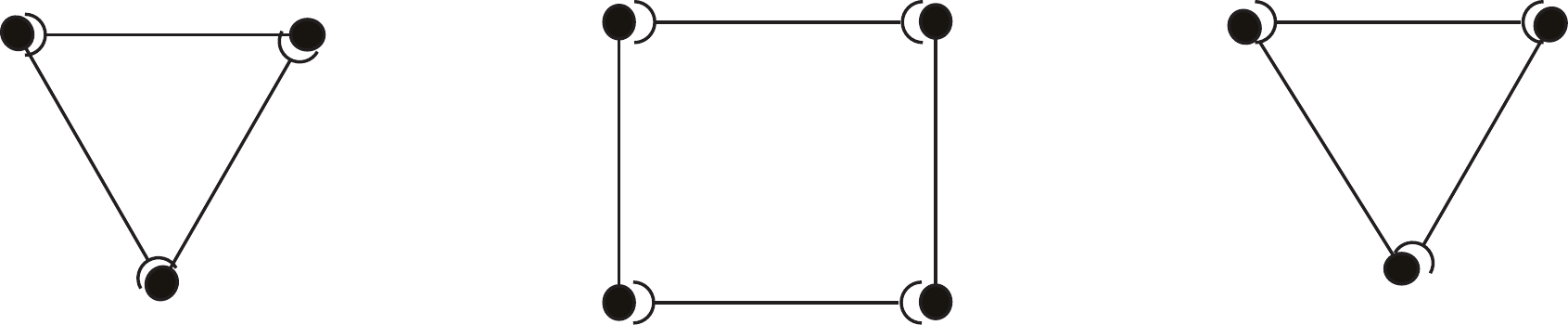}
    \caption{Tiled one-dimensional simplicial complexes.}
    \label{Fig_extile}
      \end{center}
 \end{figure}
 
 \item  In dimension 2,  $\partial\Delta_2\times [0,1]$ can be tiled using six $T_1^2$, by gluing three copies of the tiling shown in Figure~3.
 
 \begin{figure}[h]
   \begin{center}
    \includegraphics[scale=0.45]{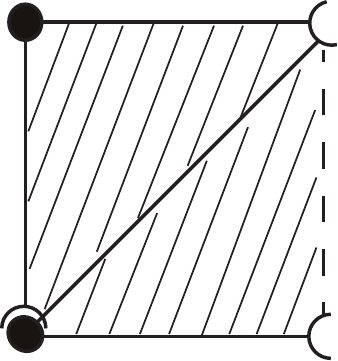}
    \caption{A union of two $T_1^2$.}
    \label{Fig_twoP}
      \end{center}
 \end{figure}
 
 \item  The prism $\partial(\Delta_2\times [0,1])$ can be tiled using six $T_1^2$ and two $T^2_3$, by gluing on the two boundary components of the previous example the open simplices $T^2_3$.

 \item The prism $\partial(\Delta_2\times [0,1])$ can likewise be tiled by  six $T_2^2$ and two $T^2_0$.
 
 \item The cylindrical parts in examples 3 and 4 above can be glued together to produce a tiled two-torus. 
 
\end{enumerate}
\end{ex}

\begin{defn} A tiling of an $n$-dimensional simplicial complex is called regular if and only if it uses one tile  $T_0^n$ for each connected component. 
\end{defn}

Recall that a simplicial complex is called \emph{shellable} if its maximal simplices can be arranged in linear order $\sigma_1,\sigma_2,\ldots,\sigma_t$ in such a way  that the subcomplex ($\cup_{i=1}^{k-1}\sigma_i)\cap \sigma_k$  is pure and $(\dim \sigma_k-1)$-dimensional for all $k=2,\ldots, t$, see \cite{Koz} for instance.  

\begin{prop} Let $K$ be a tiled finite $n$-dimensional simplicial complex that has a filtration 
$K_1\subset K_2\subset\ldots\subset K_{N-1}\subset K_N=K$ of tiled subcomplexes such that for every $i\in \{1,\ldots N\},$
$K_i$ contains $i$ tiles. Then, the tiling of $K$ is regular and  its connected components are  shellable and homotopy equivalent  to  $n$-dimensional  spheres or  balls.
\end{prop}

(Shellable simplicial complexes are known to be homotopy equivalent to wedges
of spheres, see Theorem~12.3 of \cite{Koz} for instance). 

\begin{proof}
We proceed by induction on $N$. If $N=1$, $K=T_0^n=\Delta_n,$ so that $K$ has a regular tiling and is homeomorphic to a ball.
Let us suppose that the result holds true for $N$ and that $K$ has a filtration $K_1\subset \ldots \subset K_N\subset K_{N+1}=K$.
Then, $K_{N+1}=K_N\cup T$ where $T\in \{T_0^n,\ldots, T_{n+1}^n\}$, and $K_N$ is regular by the hypothesis. If $K_N$ and $K_{N+1}$ do not have the same number of connected components, then $T$ is a connected component   of $K$ so that   $T=T_0^n$ as in the case $N=1$ and $K_{N+1}$ is regular. 
Moreover, $K_N$ being homotopy  equivalent to a union of  $n$-dimensional spheres or balls,  so is $K_{N+1}$, with one more ball component. Otherwise, $K_N$ and $K_{N+1}$ have the same number of  connected components. In this case, $T$ is glued to one of the connected components of $K_N$  and thus $T\neq T_0^n$.  Hence, $K_{N+1}$ is regular, which proves the first part. Moreover, if $T=T^n_{n+1}=\stackrel{\circ}{\Delta}$, then since $K_N$ is the union of spheres or balls so is $K_{N+1}$, a homotopy ball of $K_N$ becoming a homotopy sphere. If $T\neq T^n_{n+1}$, the boundary of $T$ is not empty so that $T$ intersects the boundary of $K_{N+1}$. In this case, $K_{N}$ is a deformation retract of $K_{N+1}$ and is homotopy equivalent to $K_{N+1}.$  Hence the result.
\end{proof}

\begin{defn} 
Let $K$ be an $n$-dimensional finite simplicial complex equipped with a tiling $\tT$. For every $i\in \{0,\ldots,n+1\}$, let $h_i(\tT)$ be the number of tiles $T_i^n$ of $\tT.$ The vector $(h_0(\tT),\ldots, h_{n+1}(\tT))$ is called the $h$-vector  of $\tT$ and the polynomial
$h_{\tT}(X)=\sum_{i=0}^{n+1}h_i(\tT)X^i$  its $h$-polynomial.
\end{defn}

The following Theorem~\ref{Thm_fvsh} and Corollary~\ref{Cor_hvect} complete Theorem~\ref{Thm_hvTile}. 
\begin{thm}\label{Thm_fvsh}
Let $K$ be a tileable $n$-dimensional finite simplicial complex. For every tiling $\tT$ of $K$,  its $h$-polynomial satisfies  $\sum_{i=0}^{n+1}h_i(\tT)X^{n+1-i}=\sum_{i=0}^{n+1}f_{i-1}(K)(X-1)^{n+1-i}$ provided $f_{-1}(K)$ is chosen to be equal to $h_0(\tT).$  In particular, two tilings $\tT$ and $\tT'$ have the same $h$-polynomial if  and only if $h_0(\tT)=h_0(\tT').$
\end{thm}
\begin{proof}
Let $\tT$ be a tiling of $K$. From Proposition~\ref{Prop_fi}, 
$$\begin{array}{rcl}
\sum_{j=0}^{n+1}f_{j-1}(K)T^j&=&\sum_{j=0}^{n+1}T^j\sum_{s=0}^{^j}\binom{n+1-s}{j-s}h_s(\tT)\\
&=&\sum_{k=0}^{n+1} h_{k}(\tT)\sum_{j=s}^{n+1}\binom{n+1-s}{j-s}T^j\\
&=&\sum_{s=0}^{n+1}h_s(\tT)\sum_{j=0}^{n+1-s}\binom{n+1-s}{j} T^{n+1-j}\\
&=&T^{n+1} \sum_{s=0}^{n+1} h_s(\tT)(1+\frac{1}{T})^{n+1-s},
\end{array}
$$
where the second line follows from Proposition~\ref{Prop_fi} and the convention $f_{-1}(K)=h_0(K)$.

By letting $X=\frac{T+1}{T}$ or equivalently $T=\frac{1}{X-1}$, we get $$\sum_{s=0}^{n+1}h_s(\tT)X^{n+1-s}=\sum_{j=0}^{n+1}f_{j-1}(K)(X-1)^{n+1-j}.$$ Hence the result.
\end{proof}

\begin{cor}\label{Cor_hvect}
Let $K$ be a connected $n$-dimensional finite simplicial complex  equipped with a regular tiling $\tT.$ Then, the $h$-vector of $\tT$ coincides with the $h$-vector of $K$.
\end{cor}

\begin{proof} Since $K$ is connected and $\tT$ regular, $h_0(\tT)=1$. But the $h$-polynomial of $K$ is by definition the polynomial  satisfying the relation in Theorem~\ref{Thm_fvsh} with $f_{-1}=1$.  Hence the result.
\end{proof}

\begin{rem} Corollary~\ref{Cor_hvect} provides a geometric  interpretation for  the  $h$-vector of tiled simplicial complexes. It is similar to the known one for shellable complexes.
\end{rem}

\begin{cor}\label{Cor_Chi}
Let $\tT$ be a tiling of  an $n$-dimensional finite simplicial complex $K$. Then $\chi(K)=h_0(\tT)+(-1)^nh_{n+1}(\tT).$
\end{cor}

\begin{proof}
 From Theorem~\ref{Thm_fvsh}, $\sum_{i=0}^{n+1}h_i(\tT)X^{n+1-i}=\sum_{i=0}^{n+1}f_{i-1}(K)(X-1)^{n+1-i}$ with $f_{-1}=h_0(\tT)$. By letting $X=0$ we get $h_{n+1}(\tT)=(-1)^{n+1}h_0(\tT)+(-1)^n\chi(K)$ as $\chi(K)=\sum_{i=0}^{n} (-1)^if_i(K).$ Hence the result.
\end{proof}

The Euler characteristic  of an even dimensional finite tiled simplicial complex is thus positive.
Corollary~\ref{Cor_dDelta}  provides a tiling  of spheres in any dimensions and Example~5 a tileable triangulation of a two-torus. Which three-manifolds possess tileable triangulations?

\subsection{Barycentric subdivision}

Recall that for every $n>0$ and $s\in\{0,\ldots, n+1\},$
$T^n_s=\Delta_n\setminus (\sigma_1\cup \ldots\cup \sigma_s)$ where $\sigma_i$ denotes a facet of $\Delta_n.$
We set $\Sd(T_s^n)=\Sd(\Delta_n)\setminus \cup_{i=1}^s\Sd({\sigma_i})$. The remaining part of the paper rely on the following key result.

\begin{thm}\label{Thm_SdPieces}
For every  $n>0$ and every $s\in\{0,\ldots, n+1\}$, $\Sd(T_s^n)$ is tileable.
Moreover, it can be tiled in such a way that only $\Sd(T_0^{n})$ (resp. $\Sd(T_{n+1}^n)$) contains the tile $T_0^{n}$ (resp. $T_{n+1}^n$) in its tiling and it contains exactly one such tile.
\end{thm}

\begin{proof} 
We proceed by induction on the dimension $n>0$. If $n=1$, the tilings $\Sd(T_0^1)=T_0^1\sqcup T_1^1$, $\Sd(T_1^1)=2T_1^1$ and $\Sd(T_2^1)=T_1^1\sqcup T_2^1$ are suitable,  see Figure~\ref{Fig_Sdtile}.

 \begin{figure}[h]
   \begin{center}
    \includegraphics[scale=0.5]{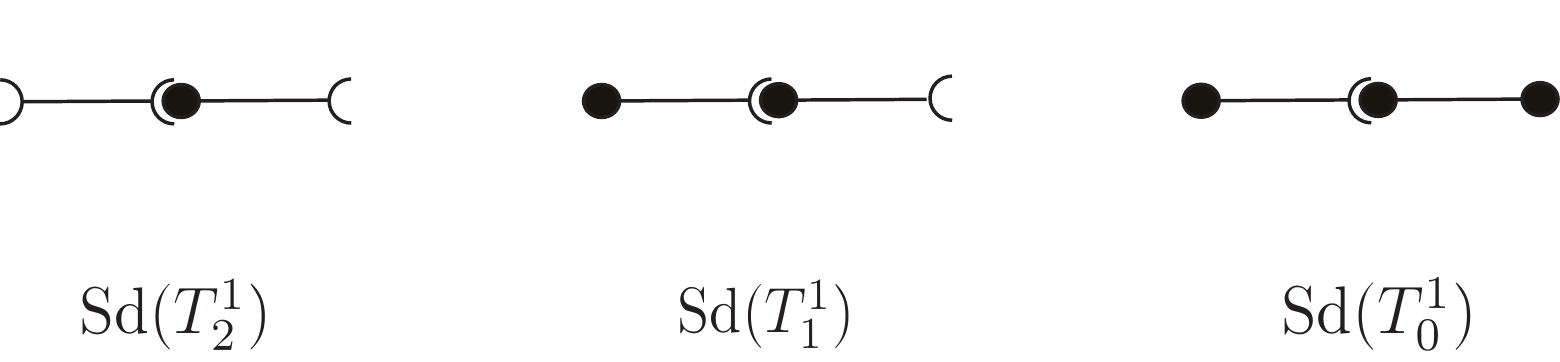}
    \caption{Tilings of subdivided one-dimensional tiles.}
    \label{Fig_Sdtile}
      \end{center}
 \end{figure}

Now, let us assume that the result holds true for $r\leq n$ and let us prove it for $r=n+1$.  From Corollary~\ref{Cor_dDelta}, $\partial \Delta_{n+1}$ has a tiling $\bigsqcup_{s=0}^{n+1}T_s^n$. We equip $\Sd(\partial \Delta_{n+1})=\bigsqcup_{s=0}^{n+1}\Sd(T_s^n)$  with the regular tiling given by the induction hypothesis. Then, $\Sd(\Delta_{n+1})$ gets a partition by cones over the tiles of $\Sd(\partial \Delta_{n+1})$ centered at the barycenter of $\Delta_{n+1}$ where all the cones except the one over $T_0^n$  are deprived of their center. From Proposition~\ref{Prop_Pieces},   this partition induces a regular tiling of $\Sd(\Delta_{n+1})=T_0^{n+1}$. For every $s\in \{1, \ldots, n+2\},$ we equip $\Sd(T_s^{n+1})=\Sd(\Delta_{n+1})\setminus \sqcup_{j=0}^{s-1}\Sd(T_j^n)$ with the tiling induced by  removing the bases of all the cones over the tiles
 $\bigcup_{j=0}^{s-1}T_j^n\subset \Sd(\partial \Delta_{n+1}).$
From Proposition~\ref{Prop_Pieces},  these cones deprived of their bases are tiles so that we get as well a tiling of $\Sd(T_s^{n+1})$. Moreover, when $s>0$,  the cone over the unique tile $T_0^n$ of the tiling $\Sd(\partial \Delta_{n+1})$ is deprived of its basis, so that the tiling we get  does not contain  $T_0^{n+1}.$  Finally, by the induction hypothesis the tiling of $\Sd(\partial \Delta_{n+1})$ contains a unique tile $T_{n+1}^n$ which is contained in the tiling of $\Sd(T_{n+1}^n)\subset \Sd(\partial \Delta_{n+1}).$  Thus, the tiling of $\Sd(T_s^{n+1})$ contains  the tile $T_{n+2}^{n+1}$ only when $s=n+2$ and in this case it contains only one such tile, since from Proposition~\ref{Prop_Pieces}, $T_{n+2}^{n+1}$ is the cone over $T_{n+1}^{n}$ deprived of its base and its center. Hence the result. 
\end{proof}

The proof of Theorem~\ref{Thm_SdPieces} provides a tiling of  all the subdivided tiles $\Sd(T_s^n)$, $s\in\{0,\ldots, n+1\}$,  in any dimension $n$. Let $H_n$ be the $(n+2)\times(n+2)$ matrix whose $(s+1)$-{st} row is the $h$-vector of the tiling of $\Sd(T_s^n)$.  It follows from the proof of Theorem~\ref{Thm_SdPieces} that the  $s$-{th} row of $H_{n+1}$ is obtained by adding the first $s-1$ rows of $H_n$ shifted by one step to the right  to the $n+3-s$ last rows of $H_n.$

\begin{ex} The matrices $H_n$ for  $n\leq 5$ are the following.

$H_1=\left(\begin{array}{ccc}1&1&0\\0&2&0\\0&1&1\end{array}\right)$,
 $H_2=\left(\begin{array}{cccc}1&4&1&0\\0&4&2&0\\0&2&4&0\\0&1&4&1\end{array}\right)$,
 $H_3=\left(\begin{array}{ccccc}1&11&11&1&0\\0&8&14&2&0\\0&4&16&4&0\\0&2&14&8&0\\0&1&11&11&1\end{array}\right)$,

 $H_4=\left(\begin{array}{cccccc}1&26&66&26&1&0\\0&16&66&36&2&0\\0&8&60&48&4&0\\0&4&48&60&8&0\\0&2&36&66&16&0\\0&1&26&66&26&1\end{array}\right),$
$H_5=\left(\begin{array}{ccccccc}1&57&302&302&57&1&0\\0&32&262&342&82&2&0\\0&16&212&372&116&4&0\\0&8&160&384&160&8&0\\0&4&116&372&212&16&0\\0&2&82&342&262&32&0\\0&1&57&302&302&57&1\end{array}\right).$
\end{ex}

We thus set $H_0=\left(\begin{array}{cc}1&0\\0&1\end{array}\right)$, but do not consider the case $n=0$ throughout the paper.

\begin{cor}\label{Cor_SdK}
Every tiling $\tT$ (resp. every regular tiling) of a finite simplicial complex $K$ induces a tiling $\Sd(\tT)$ (resp. a regular tiling) on $\Sd(K)$. Moreover,
$h(\Sd(\tT))=h(\tT)^t H_n$.  $\square$
 \end{cor}

Let $\Lambda_n=[\lambda_{ij}]$ be the lower triangular  $(n+2)\times(n+2)$ matrix where $\lambda_{ij}$ denotes the number of $j-2$ interior faces of $\Sd(\Delta_{i-2})$, $1\leq i,j\leq n+2$. We agree that the standard simplex of dimension  $-1$ has a unique face in dimension $-1$. The diagonal entries of $\Lambda_n$ are $(i-1)!$, $i\in \{1,\ldots, n+2\}$, see \cite{BW,DPS, SW1}. Let $F_n$ be the  $(n+2)\times(n+2)$ matrix whose  $s$-th row is the face vector $(f_{-1}(T^n_{s-1}),\ldots, f_n(T^n_{s-1}))$  of $T_{s-1}^{n}$, where we set $f_{-1}(T_s^n)=0$ if $s\neq 0$ and $f_{-1}(T_0^n)=1$.  

\begin{thm}\label{Thm_HFL}
For every $n>0,$ $F_n$ is unipotent upper triangular and $H_nF_n=F_n\Lambda_n$. Moreover, $H_n=[h_{ij}]_{1\leq i,j\leq n+2}$ satisfies the symmetry property $h_{ij}=h_{n+3-i, n+3-j}$.
\end{thm}

\begin{proof}
It follows from Proposition~\ref{Prop_fi} that $F_n=[\binom{n+2-i}{n+2-j}]_{1\leq i,j\leq n+2}$, thus it is a unipotent upper triangular matrix.
Now, by definition of $H_n$ and $F_n$, the  $s$-th row of the product $H_nF_n$
is the face vector of $\Sd(T_{s-1}^n)$, $s\in \{1,\ldots,n+2\}.$ Likewise, by definition of $F_n$ and $\Lambda_n$, the  $s$-th row of the product $F_n\Lambda_n$
is the face vector of $\Sd(T_{s-1}^n)$, $s\in \{1,\ldots,n+2\}.$  Thus, the two products coincide. 

Finally, to prove the symmetry property of $H_n$ we proceed by induction. The matrix $H_1$ satisfies this symmetry.  Let us now assume that  $H_r$ satisfies the symmetry property for every $r\leq n$. By definition, the  $s$-th row of $H_{n+1}$ is obtained by adding up the first $s-1$ rows of $H_n$ shifted once to the right with its $n+3-s$ last rows. 
We deduce that for every  $1\leq s,j\leq n+3,$

$$\begin{array}{rcl}
h^{n+1}_{sj}&=&\sum_{i=1}^{s-1}h_{i j-1}^n+\sum_{i=s}^{n+2} h_{ij}^n\\
&=&\sum_{i=1}^{s-1} h^n_{n+3-i, n+4-j} +\sum_{i=s}^{n+2}h^n_{n+3-i, n+3-j},
\end{array}$$
since $H_n$ satisfies the symmetry property by induction hypothesis. Thus,
$$\begin{array}{rcl}
h^{n+1}_{sj}&=&\sum_{i=n+4-s}^{n+2} h^n_{i,n+4-j}+\sum_{i=1}^{n+3-s} h^n_{i,n+3-j}\\
&=&h^{n+1}_{n+4-s, n+4-j}.
\end{array}$$
Hence the result.
\end{proof}

Let $\rho_n$ be the involution $(h_0,\ldots, h_{n+1})\in \R^{n+2}\mapsto (h_{n+1},\ldots,h_0)\in \R^{n+2}.$ The symmetry property given by Theorem~\ref{Thm_HFL}
means that the endomorphism $H_n$ of $\R^{n+2}$ commutes with  $\rho_n.$

\begin{cor}\label{Cor_Rho} For every $n>0,$ $H_n$ is diagonalizable with eigenvalues $s!$, $s\in\{0,1,\ldots, n+1\}$. Moreover, the restriction of $\rho_n$  to  the eigenspace of $s!$ is  $(-1)^{n+1-s}\textup{id}$. The vector $(1,1,\ldots,1)=(f_n(T_0^n),\ldots, f_n(T^n_{n+1}))$  spans the eigenspace of $H_n$ associated to the eigenvalue $(n+1)!.$ 
\end{cor}

\begin{proof}
By Theorem~\ref{Thm_HFL}, $H_n$ is conjugated to the matrix $\Lambda_n$ which is diagonalizable with eigenvalues $s!, s\in \{0,\ldots, n+1\}$, see \cite{BW, DPS, SW1}. The first part follows. Again by Theorem~\ref{Thm_HFL} we know that the last column of $F_n$  is an eigenvector of $H_n$ associated to the eigenvalue $(n+1)!.$ It is the vector  $(f_n(T_0^n),\ldots, f_n(T_{n+1}^n))=(1,\ldots,1)$ which is preserved by $\rho_n.$ The result is thus proved for $n=1$, since 1 is an  eigenvalue of multiplicity two and the signature of $\rho_n$ vanishes on the corresponding eigenspace. We will prove by induction that if $\rho_n$ acts as $-\id$ on the eigenspace associated to the eigenvalue $s!$ of $H_n$, then  $\rho_{n+1}$
acts as $+\id$ on the eigenspace associated to the eigenvalue $s!$ of $H_{n+1}$. The result then follows, as the signature of $\rho_n$ is either 0 or 1 depending on the parity of $n$. Let $s\leq n$ be such that $\rho_n$ acts as $-id$ on the eigenspace associated to the eigenvalue $s!$ of $H_n$. By Theorem~\ref{Thm_HFL}, the eigenvectors associated to this eigenvalue are linear combination of the $n+2-s$ last columns of $F_n$.  Indeed, let $X_{n,s}$ be an eigenvector of $\Lambda_n$  associated to the eigenvalue $s!$, then $F_n X_{n,s}$ is an eigenvector of $H_n$ corresponding to $s!$. By hypothesis, $\rho_n$  acts as $-\id$ on the eigenspace spanned by this vector, so that $(1,\ldots,1)^tF_nX_{n,s}=0$. By definition of $F_n$ and from Corollary~\ref{Cor_dDelta}, $(1,\ldots,1)^tF_n$ is the face vector $f_n$ of $\partial \Delta_{n+1}.$ As $\Lambda_n$ is lower triangular, there exists $x_{n+1}\in \R$ such that $X_{n+1,s}=\left(\begin{array}{c}X_{n,s}\\x_{n+1}\end{array}\right)$. The first row of $F_{n+1}$ is the face vector of $\Delta_{n+1}=T_0^{n+1}$ and thus equals $(f_n^t,1).$  The first coefficient of $F_{n+1}X_{n+1,s}$ is thus $f_n^tX_{n,s} +x_{n+1}=x_{n+1}.$  However, the last row of $F_{n+1}$ is the face vector of $\stackrel{\circ}{\Delta}_n=T_{n+2}^{n+1}$ so that the last coefficient of $F_{n+1}X_{n+1,s}$ is  $x_{n+1}$ as well.  Similarly,  the second row of $F_{n+1}$ is the face vector of $T_1^{n+1}$. From Proposition~\ref{Prop_Pieces}, $T_0^{n+1}=T_1^{n+1}\cup T_0^n$ so that  this second row differs from the first one by the face vector of $T_0^n$. Let us denote by $F_n^i$ the rows of $F_n$, $i\in \{1,\ldots,n+2\}$.  We deduce that  $F_{n+1}^2X_{n+1,s}=F^1_{n+1}X_{n+1,s}-F^1_nX_{n,s}=x_{n+1}-F^1_nX_{n,s}$. Since by Proposition~\ref{Prop_Pieces} we have $T_{n+1}^{n+1}=T_{n+2}^{n+1}\cup T_{n}^{n+1}$, we deduce that $F_{n+1}^{n+2}X_{n+1,s}=F^{n+3}_{n+1}X_{n+1,s}+F^{n+2}_nX_{n,s}=x_{n+1}+F^{n+2}_nX_{n,s}.$ By the induction hypothesis, the eigenvector $F_n X_{n,s}$  is reversed by $\rho_n$ so that $F^1_nX_{n,s}=-F_n^{n+2}X_{n,s}$.  Therefore, the second coefficient of the eigenvector $F_{n+1}X_{n+1,s}$ coincides with its second to last. Proceeding in the same way by induction, we deduce that the $i$-th coefficient of $F_{n+1}X_{n+1,s}$ coincides with the $(n+4-i)$-th one for $i\in\{1,\ldots,n+3\},$ which means $\rho_{n+1}(F_{n+1}X_{n+1,s})=F_{n+1}X_{n+1,s}$. Hence the result.\end{proof}

Let $h^n=(h_0^n,\ldots, h^n_{n+1})$ be the eigenvector of the transposed matrix $H_n^t$ associated to the eigenvalue $(n+1)!$ and normalized in such a way  that   $\sum_{s=0}^{n+1}h^n_s=1$. For every tiled finite simplicial complex $(K,\tT)$ of dimension $n$, we set $|h(\tT)|=\sum_{s=0}^{n}h_s(\tT)$.

\begin{cor}\label{Cor_$h^n$}
For every tiled finite  $n$-dimensional simplicial complex $(K, \tT)$, the sequence $\frac{1}{|h(\tT)|(n+1)!^d}h(\Sd^d(\tT))$ converges to $h^n$ as $d$ grows to $+\infty.$ Moreover, $h_0^n=h^n_{n+1}=0$ and $h^n$ is preserved by the symmetry $\rho_n$. Finally, $\frac{1}{(n+1)!^d}H_n^d$ converges to the matrix $(1,1,\ldots,1)(h^n)^t$ as $d$ grows to $+\infty.$ 
\end{cor}

\begin{proof}
From Corollary~\ref{Cor_Rho}  we know that $H_n=PDP^{-1}$, where $P$ denotes the matrix of eigenvectors of $H_n$ and $D$ the diagonal  matrix  of eigenvalues $s!, s\in\{0,\ldots,n+1\}$, so that $(P^{-1})^t$ is a matrix of eigenvectors of $H_n^t$. Thus, $\frac{1}{(n+1)!^d}H_n^d=P\frac{D^d}{(n+1)!^d}P^{-1}$ and $\frac{D^d}{(n+1)!^d}$ converges to $\textup{Diag}(0,\ldots,0,1)$ as $d$ grows to $+\infty$, compare \cite{DPS}. Therefore, $\frac{H_n^d}{(n+1)!^d}$ converges to the product $(1,\ldots, 1)^t h$, where $h$ denotes an eigenvector of $H_n^t$ associated to the eigenvalue $(n+1)!$, since $(1,\ldots,1)$ is an eigenvector of $H_n$ associated to the  eigenvalue $(n+1)!$ from Corollary~\ref{Cor_Rho}. For every $d>0$, $(1,\ldots, 1)$  is preserved by $\frac{H_n^d}{(n+1)!^d}$, so that it is also preserved in the limit by $(1,\ldots,1)h^t$. Thus, $|h|=1$ and since the eigenvalue $(n+1)!$ is simple from Corollary~\ref{Cor_Rho}, we deduce that $h=h^n$. From Corollary~\ref{Cor_SdK} we know that for every $d>0,$ $h(\Sd^d(\tT))=h(\tT)^t H_n^d.$  We deduce that $\frac{1}{|h(\tT)|(n+1)!^d}{h(\Sd(\tT))}$ converges to $\frac{1}{|h(\tT)|} h(\tT)^t(1,\ldots,1) {(h^n)}^t=h^n.$  The fact that $h^n$ is preserved by $\rho_n$ follows from Corollary~\ref{Cor_Rho}, since $h^n$ is an eigenvector of $H_n^t$ associated to the eigenvalue $(n+1)!$ and the matrix $J$ of $\rho_n$ is symmetric. Indeed,
Corollary~\ref{Cor_Rho} implies that  $J P=P \textup{Diag}((-1)^{n+2-i})$, so that $P^{-1}J=\textup{Diag}((-1)^{n+2-i}) P^{-1}$ and  $J^t(P^{-1})^t= (P^{-1})^t\textup{Diag}((-1)^{n+2-i}) $. As $J^t=J$ and the last column of $(P^{-1})^t$ is $h^n$, $\rho_n$ fixes $h^n$.  (Another way to see that $\rho_n(h^n)=h^n$ is to consider  a basis $(e_0,\ldots, e_{n+1})$ of eigenvectors of $H_n$. Then the dual basis  $(e^*_0,\ldots, e^*_{n+1})$ is  made of eigenvectors of $H_n^t$. If $\rho_n e_i=\epsilon_i e_i$, then  $\rho_n^t e^*_i=\epsilon_i e^*_i$, with $\epsilon_i \in\{\pm1\}$ and $\rho_n^t=\rho_n$ under canonical identification between $(\R^{n+2})^*$ and $\R^{n+2}.$) Finally, by induction on $d$ we deduce from Theorem~\ref{Thm_SdPieces} that  the number of tiles $T_0^n$ and  $T^n_{n+1}$ which are in the tiling $\Sd(\tT)$ does not depend on $d$ so that  $h_0^n=\lim_{d\to +\infty}\frac{1}{|h(\tT)|(n+1)^d}{h_0(\Sd^d(\tT))}=0$ and  $h_{n+1}^n=\lim_{d\to +\infty}\frac{1}{|h(\tT)|(n+1)^d}{h_n(\Sd^d(\tT))}=0$. \end{proof}

Recall that for every finite $n$-dimensional simplicial complex $K$ with face polynomial $q_K(T)=\sum_{p=0}^{n}f_p(K)T^p,$ the polynomial $\frac{1}{f_n(K)(n+1)!^d}q_{\Sd^d(K)}(T)$ converges to a limit polynomial $q^\infty_n=\sum_{p=0}^{n} q_{p,n}T^p$, see \cite{BW,DPS,SW1}. The first part of Corollary~\ref{Cor_$h^n$} is nothing but this result expressed in terms of $h$-vector and its proof is similar to the one of \cite{DPS}.

\begin{cor} For every $n>0,$ $h^n(X)=(X-1)^n q^n\big(\frac{1}{X-1}\big)$, where $h^n(X)=\sum_{i=0}^nh_i^nX^i.$
\end{cor}

\begin{proof} Let $K$ be a finite $n$-dimensional simplicial complex equipped with a tiling $\tT.$ For example, $K=\Delta_n$ being tiled with a single $T_0^n.$ From Theorem~\ref{Thm_fvsh} we know that $|h(\tT)|=f_n(K)$ and deduce $$\frac{1}{|h(\tT)|(n+1)!^d}\sum_{i=0}^{n+1}h_i(\Sd^d(\tT))X^{n+1-i}=\frac{1}{f_n(K)(n+1)!^d}\sum_{i=0}^{n+1}f_{i-1}(\Sd^d(K))(X-1)^{n+1-i}$$ with $f_{-1}(\Sd^d(K))=h_0(\Sd^d(\tT))$. From Theorem~\ref{Thm_SdPieces}, $h_0(\Sd^d(\tT))=h_0(\tT).$ We thus deduce from Corollary~\ref{Cor_$h^n$} and \cite{DPS} by passing to the limit as $d$ grows to $+\infty$ that  $\sum_{i=0}^{n+1}h_i^nX^{n+1-i}=\sum_{i=1}^{n+1}q_{i-1,n}(X-1)^{n+1-i}=(X-1)^nq^n(\frac{1}{X-1})$. The result now follows from Corollary~\ref{Cor_$h^n$}, since $h_0^n=h_{n+1}^n=0$  and $\rho_n(h^n)=h^n$.
\end{proof}

\section{Packings}
\begin{thm}\label{Thm_Packing}
For every $n>0$,  $\Sd(\Delta_n)$ contains a packing of disjoint simplices with one $n$-simplex and for every  $j\in \{0,\ldots, n-1\}$, $2^{n-1-j}$ $j$-simplices. Moreover, for every $s\in\{1,\ldots, n+1\}$, $\Sd(T_s^n)$ contains a packing of disjoint simplices with one $(n+1-s)$-simplex and if $s\leq n-1$, for every $j\in\{0,\ldots,n-1-s\},$ $2^{n-1-s-j}$ simplices of dimension $j.$ The simplex of dimension $n+1-s$ reads $[\hat{\sigma}_{s-1},\ldots, \hat{\sigma}_{m}],$ where for $i\in\{s-1,\ldots, n\}$,
$\sigma_i$ denotes an $i$-simplex of $\Delta_n$.
\end{thm}

\begin{proof} We proceed by induction on the dimension $n>0$. If $n=1$, one checks the result, $\Sd(\Delta_1)\supset \Delta_1\sqcup \Delta_0, \Sd(T_1^1)\supset \Delta_1$ and $\Sd(T_2^1)\supset \Delta_0.$
Suppose now that the result holds true for every dimension $\leq n$. From Corollary~\ref{Cor_dDelta}, we know that $\partial \Delta_{n+1}$ is tileable and $\partial \Delta_{n+1}=\sqcup_{s=0}^{n+1}T_s^n$, so that $\Sd(\partial \Delta_{n+1})=\sqcup_{s=0}^{n+1}\Sd(T^n_s).$
The union of the packings given by the induction hypothesis provides a packing of the boundary of $\Sd(\Delta_{n+1})$ which contains two $n$-simplices. We replace the $n$-simplex contained in $\Sd(T_0^n)$ by its cone centered at the barycenter of $\Delta_{n+1}. $ We get in this way a packing of disjoint simplices in $\Sd(\Delta_{n+1})$ containing a simplex of dimension $n+1$, a simplex of dimension $n$ and $ 1+\sum_{s=0}^{n-1-j}2^{n-1-s-j}$ simplices of dimension $j$, $j\in \{0,\ldots,n-1\}$. Now, $ 1+\sum_{s=0}^{n-1-j}2^{n-1-s-j}=1+\sum_{s=0}^{n-1-j}2^s=2^{n-j}$.

Likewise, from Proposition~\ref{Prop_Pieces} we deduce that for every $s\in\{1,\ldots, n+1\}$, $\Sd(\partial T_{s}^{n+1})=\sqcup_{l=s}^{n+1}\Sd(T_l^n).$ The union of the  packings  given by the induction hypothesis provides a packing of simplices  in $\Sd(T_s^{n+1})\cap \Sd(\partial \Delta_{n+1})$ which contains one simplex of dimension $n+1-s$, of the form $[\hat{\sigma}_{s-1}\ldots, \hat{\sigma}_{n}],$ that we replace by its cone centered at the barycenter of $\Delta_{n+1}$.
We thus get a packing of disjoint simplices in  $\Sd(T_{s}^{n+1})$ which consists of 
one simplex of dimension $n+2-s,$ of the form $[\hat{\sigma}_{s-1},\ldots, \hat{\sigma}_{n+1}],$
and one simplex of dimension $n-s$  if $s<n+1$ together with, if $s\leq n-1$, $1+\sum_{l=0}^{n-1-s-j}2^{n-1-s-j-l}=2^{n-s-j}$ simplices of dimension $j$ for every $j\in\{0,\ldots,n-1-s\}.$
Finally, $\Sd(T_{n+2}^{n+1})$ contains the barycenter of $\Delta_{n+1},$ which can be written as $\hat{\sigma}_{n+1}.$ Hence the result.
\end{proof}

 We now able to prove Theorem~\ref{Thm_h0h1}.

\begin{proof}[Proof of Theorem~\ref{Thm_h0h1}.]
  By Corollary~\ref{Cor_SdK},  $\Sd(K)$ is tiled by $\Sd(\tT)$ so that the union of the packings given by Theorem~\ref{Thm_Packing} provides the result.
\end{proof}
  
  We may relax the condition to be disjoint in Theorems~\ref{Thm_Packing} and \ref{Thm_h0h1}, to get the following results.
  
  \begin{thm} \label{Thm_Pack}For every $n>0$, every $p\in\{1,\ldots, n-1\}$ and every $s\in \{0,\ldots, n+1\},$ $\Sd(T^n_s)$  contains a packing of simplices with  one simplex of dimension $n+1-s+p$ if $s\geq p+1$ or $2^{p-s}$ simplices of dimension $n$ if $s\leq p$ together with  $2^{n-1-s-j+p}$ simplices of dimension $j$, $j\in\{p,\ldots, \mbox{min} (n-1-s+p, n-1)\}$ if $s\leq n-1,$ in such a way that  the intersection of two simplices of this collection is of dimension less than $p$ and the intersection of each simplex with $\Sd(\partial \Delta_n)\setminus \Sd(T_s^n)$  is of dimension less than $p$. Moreover, the $(n+1-s+p)$-simplex  is of the form $[\hat{\sigma}_{s-p-1},\ldots,\hat{\sigma}_n],$ where for every $i\in \{s-p-1,\ldots,n \}$, $\sigma_i$ is an $i$-simplex of $\Delta_n.$
  \end{thm}

 \begin{proof} We proceed by induction on $p\in\{1,\ldots,n\}.$ If $p=1$, we first check the result for $n=1$. In this case,
 $\Sd(\Delta_1)$ contains exactly two simplices of dimension one intersecting each other at the barycenter of $\Delta_1$ and intersecting $\Sd(\partial \Delta_1)$ at a vertex. This provides a suitable  packing for  $\Sd(T_0^1),\Sd(T_1^1)$ and $\Sd(T_2^1)$.  
 
 If $n>1$, we know from Proposition~\ref{Prop_Pieces}  that for every $s\in\{ 1, \ldots, n\}$, $\Sd(\partial T_s^n)=\sqcup_{l=s}^{n}\Sd(T_l^{n-1}).$ The union for $l\in\{s,\ldots, n\}$ 
 of the packings given by Theorem~\ref{Thm_Packing}  provides a packing of disjoint simplices in  $\Sd(T_s^n)\cap \Sd(\partial \Delta_n)$ which contains an $(n-s)$-simplex of the form $[\hat{\sigma}_{s-1},\ldots,\hat{\sigma}_{n-1}]$, a simplex of dimension $n-1-s$  if $s\leq n-1$ and if $s<n-1$, $2^{n-1-s-j}$ simplices of dimension $j$ for every $j\in\{0,\ldots, n-2-s\}.$
 By replacing these simplices by their cones centered at the barycenter of $\Delta_n$, we get a collection of simplices in $\Sd(T_s^n)$ which contains one  $(n+1-s)$-simplex  of the form $[\hat{\sigma}_{s-1},\ldots,\hat{\sigma}_{n}]$  and if $s\leq n-1$, $2^{n-s-j}$ simplices of dimension $j$ for every $j\in\{1,\ldots, n-s\}.$
 Moreover, two simplices of this collection intersect at the barycenter of $\Delta_n$ and these simplices are contained in $\Sd(T^n_s).$ If $s\geq 2$, we remark that an $(n+1-s)$-simplex is of the form $[\hat{\sigma}_{s-1},\ldots,\hat{\sigma}_{n}]$, where $\sigma_i$ is an $i$-simplex of $\Delta_n$. By choosing a facet $\sigma_{s-2}$ of $\sigma_{s-1}$ and by replacing this $(n+1-s)$-simplex by the $(n+2-s)$-simplex $[\hat{\sigma}_{s-2},\ldots,\hat{\sigma}_{n}]$ we get the required packing. This last simplex indeed intersects $ \Sd(\partial \Delta_n)\setminus \Sd(T_s^n)$ at the vertex $\{\hat{\sigma}_{s-2}\}.$
 If $s=0$,  we likewise know from Proposition~\ref{Prop_Pieces} that  $ \Sd(\partial \Delta_n)=\sqcup_{l=0}^{n} \Sd(T_l^{n-1}).$ The union for $l\in \{0,\ldots,n\}
$ of the packings given by Theorem~\ref{Thm_Packing}  provides a packing of disjoint simplices in $\Sd(\partial \Delta_n)$ which contains two simplices of dimension $n-1$ and for every $j\in\{0,\ldots,n-2\}$, $2^{n-1-j}$
simplices of dimension $j$. By replacing these simplices by their cones centered at the barycenter of $\Delta_n$, we get a collection of simplices of $\Sd(T_0^n)$ that consists of two simplices  of dimension $n$ and of $2^{n-j}$ simplices of dimension $j$, $j\in\{1,\ldots, n-1\}$. Moreover, two such simplices intersect at the barycenter of $\Delta_n$. Finally, if $s=n+1$, the 1-simplex $[\hat{\sigma}_{n-1},\hat{\sigma}_{n}]$, where $\sigma_i\subset \Delta_n$ is of dimension $i$, intersects $\Sd(\partial \Delta_n)$ at the vertex $\{\hat{\sigma}_{s-1}\}$. This simplex gives the desired collection   and the result follows for $p=1$ and every $n>0$.

Now, let us suppose that the result holds true for every  $r\leq p-1$ and $n\geq p-1$ and let us prove it for $r=p$ and $n\geq p$. 
Let $2\leq p\leq n$ and $s\in\{0,\ldots, n\}$.
From Proposition~\ref{Prop_Pieces}, $\Sd(\partial T_s^n)=\sqcup_{l=s}^{n}\Sd(T_l^{n-1}).$  Let us equip each $\Sd(T_l^{n-1})$ with a packing given by the induction hypothesis applied to $p-1\leq n-1.$  The union of these packings gives a packing of simplices of $\Sd(\partial T_s^n)$ such that the intersection between two simplices is of dimension $\leq p-1$.  This packing contains $2^{n-2-s-j+p}$ simplices of dimension $j$ for every $j\in\{p-1, \min(n-2-s+p, n-2)\}$ if $s\leq n-1$ and one simplex of dimension $n-2-s+p$, one simplex of dimension $n-1-s+p$ if $s\geq p$ and $1+\sum_{l=s}^{p-1}2^{p-1-l}$ simplices of dimension $n-1$ if $s<p$. By replacing all these simplices by their cones centered at the barycenter of $ \Delta_n$, we get a packing of simplices in $\Sd(T_s^n)$
 which contains $2^{n-2-s-j+p}$ simplices of dimension $j+1$ for every   $j\in\{p-1, \min(n-2-s+p, n-2)\}$, that is to say $2^{n-1-s-j+p}$
simplices of dimension $j$ for $j\in\{p, \min(n-1-s+p, n-1)\}$ if $s\leq n-1$,   as well as $2^{p-s}$
simplices of dimension $n$ if $s<p$ and one simplex of dimension $n-s+p$ if $s\geq p$. If $s>p$, this last simplex, by construction, is of the form  $[\hat{\sigma}_{s-p},\ldots,\hat{\sigma}_{n}]$, where $\sigma_i\subset \Delta_n$ is of dimension $i$. By choosing a facet $\sigma_{s-p-1}$ of $\sigma_{s-p}$ we replace this $(n-s+p)$-simplex by the $(n+1-s+p)$-simplex $[\hat{\sigma}_{s-p-1},\ldots,\hat{\sigma}_{n}]$ to get the result for $s\leq n$. Indeed, by construction and the induction hypothesis, two disjoint simplices from this packing intersect each other in dimension $\leq p-1$ and each simplex intersect $\Sd(\Delta_n)\setminus \Sd(T_s^n)$ in dimension $\leq p-1$.
If $s=n+1$, every  simplex of the form $[\hat{\sigma}_{n-p},\ldots,\hat{\sigma}_{n}]\in \Sd(\Delta_n)$, where $\sigma_i\subset \Delta_n$ is of dimension $i$, gives a $p$-simplex which intersect $\Sd(\partial \Delta_n)$  at the $(p-1)$-simplex  $[\hat{\sigma}_{n-p},\ldots,\hat{\sigma}_{n-1}]$. A packing of $\Sd(T_{n+1}^n)$ reduced to this simplex 
is suitable. Hence the result.
  \end{proof}

\begin{cor}\label{Cor_PackJoint}
Let $K$ be a finite $n$-dimensional simplicial complex equipped with a tiling $\tT$ of $h$-vector $(h_0(\tT),\ldots, h_{n+1}(\tT))$ and let $1\leq p\leq n-1$. Then, it is possible to pack $h_{p+1}(\tT)+2^{p}\sum_{s=0}^{p}\frac{h_s(\tT)}{2^s}$ simplices of dimension $n$ in $\Sd(K)$ in such a way that they intersect each other  in dimension less than $p.$ Moreover, this packing can be completed by $h_{n+1-j+p}(\tT)+2^{n-1-j+p}\sum_{s=0}^{n-1-j+p}\frac{h_s(\tT)}{2^s}$ simplices of dimension $j$  having the same property,  $j\in \{p,\ldots, n-1\}$.
\end{cor}

\begin{proof} By Corollary~\ref{Cor_SdK}, $\Sd(K)$ is tiled by $\Sd(\tT)$ so that  the union of the packings given by Theorem~\ref{Thm_Pack} provides the result.
\end{proof}

\begin{rem}
\begin{enumerate}
\item Corollary~\ref{Cor_PackJoint} remains valid in the case $p=0$ using the convention that simplices intersect in negative dimension when they are disjoint. The case $p=0$ then gives back Theorem~\ref{Thm_h0h1}.
\item For every  tiled finite $n$-dimensional simplicial complex $(K,\tT)$, e.g. $K=\Delta_n$, and every $0\leq p\leq n$, Theorem~\ref{Thm_h0h1} and Corollary~\ref{Cor_PackJoint} provide a sequence of packings in $\Sd^d(K), d>0$. Corollary~\ref{Cor_$h^n$} provides the asymptotic of the number of simplices in each dimension of this sequence.
\end{enumerate}
\end{rem}

We finally prove Theorem~\ref{Thm_mnp}. 

\begin{proof}[Proof of Theorem~\ref{Thm_mnp}.] 
 By definition, $\lambda_{p,\nu}(n)=\lim_{d\to \infty}\lambda_{p,\nu}^d(n)$ and $\lambda_{p,\nu}^d(n)=\frac{1}{(n+1)!^d}\mbox{max}_{L\in \mathcal{L}^{n,p}_d} M_{p,\nu}(L)$,  see \S~\ref{Sect_Refined}. However, $\stackrel{\circ}{\Delta}_n=T_{n+1}^n$ is tiled by a single tile and this tiling induces a tiling $\Sd^d(\tT)$ of $\Sd^d(\stackrel{\circ}{\Delta}_n)$  for every $d>0$, see Theorem~\ref{Thm_Pack}.  Corollary~\ref{Cor_PackJoint} then provides a subcomplex $L_d\in \mathcal{L}_d^{n,p}$, see \S~\ref{Sect_Refined}, which contains $h_p(\Sd^{d-1}(\tT))+2^{p-1}\sum_{i=0}^{p-1}\frac{h_i(\Sd^{d-1}(\tT))}{2^i}$  simplices of dimension $n$ together with 
$h_{n+p-j}(\Sd^{d-1}(\tT))+2^{n-2+p-j}\sum_{i=0}^{n-2+p-j}\frac{h_i(\Sd^{d-1}(\tT))}{2^i}$ simplices of dimension $j$,  $j\in\{p-1,\ldots, n-1\}$.

From Corollary~\ref{Cor_$h^n$}, we know that  for every $i\in \{0,\ldots, n\},$  $\frac{h_i(\Sd^{d-1}(\tT))}{(n+1)!^d}$ converges to $\frac{h_i^n}{(n+1)!}$ as $d$ grows to $+\infty$ and that $h_0^n=h^n_{n+1}=0.$
For every $p\in \{1,\ldots, n-1\}$ and every $j\in\{p+1,\ldots,n\},$
$M_{p,\nu}(\Delta_j)=\nu(1-\nu)(\nu^p+(1-\nu)^p)\binom{j}{p+1}$  by Lemma~\ref{Lem_Mpdelta}.  The result follows after the change of variables $j\to n+p-j.$
\end{proof}

\addcontentsline{toc}{part}{References}

\bibliography{LimiteEd}
\bibliographystyle{abbrv}

Univ Lyon, Universit\'e Claude Bernard Lyon 1, CNRS UMR 5208, Institut Camille Jordan, 43 blvd. du 11 novembre 1918, F-69622 Villeurbanne cedex, France

{salepci@math.univ-lyon1.fr, welschinger@math.univ-lyon1.fr.}
\end{document}